\documentclass{amsart}
\usepackage{amsmath}
\usepackage{amscd}
\usepackage{amssymb}
\usepackage{amsthm}
\RequirePackage{filecontents}
\usepackage{fullpage}
\usepackage[numbers]{natbib}  
\usepackage[draft]{hyperref}

\theoremstyle{plain}
\newtheorem{prop}{Proposition}
\newtheorem{lem}[prop]{Lemma}

\theoremstyle{definition}
\newtheorem{defn}[prop]{Definition}
\newtheorem{rem}[prop]{Remark}
\newtheorem{ex}[prop]{Example}

\title{Poincar\'e square series for the Weil representation}
\author{Brandon Williams }

\address{Department of Mathematics \\ University of California \\ Berkeley, CA 94720}
\email{btw@math.berkeley.edu}

\begin{document}

\nocite{*}

\maketitle

\begin{abstract} We calculate the Jacobi Eisenstein series of weight $k \ge 3$ for a certain representation of the Jacobi group, and evaluate these at $z = 0$ to give coefficient formulas for a family of modular forms $Q_{k,m,\beta}$ of weight $k \ge 5/2$ for the (dual) Weil representation on an even lattice. The forms we construct always contain all cusp forms within their span. We explain how to compute the representation numbers in the coefficient formulas for $Q_{k,m,\beta}$ and the Eisenstein series of Bruinier and Kuss $p$-adically to get an efficient algorithm. The main application is in constructing automorphic products. \end{abstract}

\section{Introduction}

Let $(V,\langle -,- \rangle)$ be a vector space of finite dimension $e = \mathrm{dim}\, V$, with nondegenerate bilinear form of signature $(b^+,b^-).$ We denote by $q(x) := \frac{1}{2}\langle x, x \rangle, \, x \in V$ the associated quadratic form. Let $\Lambda \subseteq V$ be a lattice with $q(v) \in \mathbb{Z}$ for all $v \in \Lambda.$ Recall that the \textbf{Weil representation} associated to the discriminant group $\Lambda' / \Lambda$ is a unitary representation $$\rho : \tilde \Gamma := Mp_2(\mathbb{Z}) \longrightarrow \mathrm{Aut}_{\mathbb{C}}(\mathbb{C}[\Lambda'/\Lambda])$$ defined by $$\rho(T) \mathfrak{e}_{\gamma} = \mathbf{e}\Big( q(\gamma) \Big) \mathfrak{e}_{\gamma}$$ and $$\rho(S) \mathfrak{e}_{\gamma} = \frac{\mathbf{e}((b^- - b^+)/8)}{\sqrt{|\Lambda'/\Lambda|}} \sum_{\beta} \mathbf{e} \Big( -\langle \gamma, \beta \rangle \Big) \mathfrak{e}_{\beta},$$ where $\mathfrak{e}_{\gamma}$, $\gamma \in \Lambda'/\Lambda$ is the natural basis of the group ring $\mathbb{C}[\Lambda'/\Lambda],$ and $S,T$ are the usual generators of $\tilde \Gamma;$ and $\mathbf{e}(x) = e^{2\pi i x}$. We will mainly consider the dual representation $\rho^*$. \\

Several constructions of modular forms for $\rho^*$ are known. The oldest and best-known is the theta function $$\vartheta(\tau) = \sum_{\gamma \in \Lambda'/\Lambda} \sum_{v \in \Lambda} \mathbf{e}\Big(-\tau \cdot q(\gamma + v) \Big) \, \mathfrak{e}_{\gamma},$$ which is a modular form for $\rho^*$ of weight $e/2 = b^- / 2$ when $q$ is negative definite. (We use a negative definite form $q$ to get modular forms for the dual representation.) The theta function is fundamental in the analytic theory of quadratic forms and is the motivating example for the Weil representation above. Various generalizations (for example using harmonic, homogeneous polynomials) can be used to construct other modular forms; all of these are straightforward applications of Poisson summation. \\

In \cite{BK}, Bruinier and Kuss describe a formula for the coefficients of the Eisenstein series $$E_{k,0}(\tau) = \sum_{M \in \tilde \Gamma_{\infty} \backslash \tilde \Gamma} \mathfrak{e}_0 \Big|_{k,\rho^*} M$$ when $2k - b^- + b^+ \equiv 0$ mod $4$, and $\tilde \Gamma_{\infty}$ is the subgroup of $\tilde \Gamma$ generated by $T$ and $(-I,i)$. (Note that this differs slightly from the definition in \cite{BK}, where $\tilde \Gamma_{\infty}$ is the subgroup generated by only $T$. In particular, $E_{k,0}$ will always have constant term $\mathfrak{e}_0$ in this note, rather than $2 \mathfrak{e}_0$.) \\

In this note we use methods similar to \cite{BK} to derive expressions for the coefficients of another family of modular forms for $\rho^*$, namely the ``Poincar\'e square series", which we define by $$Q_{k,m,\beta} = \sum_{\lambda \in \mathbb{Z}} P_{k,\lambda^2 m,\lambda \beta}, \; \; \beta \in \Lambda'/\Lambda, \; m \in \mathbb{Z} - q(\beta), \; m > 0,$$ where $P_{k,m,\beta}$ is the Poincar\'e series of exponential type as in \cite{Br} (and we set $P_{k,0,0} = E_{k,0}$). These are interesting because the space they span always contains all cusp forms just as $P_{k,m,\beta}$ span all cusp forms, as one can see by M\"obius inversion. (To get the entire space of modular forms, we also need to include all Eisenstein series.) In most cases, $Q_{k,m,\beta}$ is the zero-value of an appropriate Jacobi Eisenstein series. We use this fact to derive a formula for the coefficients of $Q_{k,m,\beta}$; the result is presented in section 8. Similarly to the Eisenstein series of \cite{BK}, this formula involves representation numbers of quadratic polynomials modulo prime powers; we also explain how to use $p$-adic techniques (in particular, the calculations of \cite{CKW}) to calculate them rapidly. A program in SAGE to calculate these is available on the author's university webpage. \\

The main application of these formulas is in the construction of automorphic products. Under the Borcherds lift, nearly-holomorphic modular forms (poles at cusps being allowed) for the Weil representation are the input functions from which automorphic products are constructed. Modular forms of weight $2+k$ for the dual $\rho^*$ play the role of obstructions to finding nearly-holomorphic modular forms $F$ of weight $-k$ for $\rho$, as explained in section 3 of \cite{Bo}, and we can always span all obstructions by finitely many series $Q_{k,m,\beta}$. Also, we can compute the nearly-holomorphic modular form $F$ by multiplying by an appropriate power of $\Delta$ and searching for it among cusp forms for $\rho$, which itself is the dual Weil representation for the quadratic form $-q$ and therefore is also spanned by Poincar\'e square series. This method can handle arbitrary lattices (with no restriction on the level or the dimension of the space of cusp forms). We give an example of this in section 9. \\

There are other known methods of constructing (spanning sets of) modular forms in $M_k(\rho^*)$; for example, the averaging method of Scheithauer (see for example \cite{Sch}, theorem 5.4) or an algorithm of Raum \cite{R} that is also based on Jacobi forms. However, the method described in this note seems essentially unrelated to them. The general idea of these results may already be known to experts, but the details do not seem to be readily available in the literature.\\

\textbf{Acknowledgments:} I am grateful to Richard Borcherds, Jan Hendrik Bruinier, Sebastian Opitz and Martin Raum for helpful discussions.

\tableofcontents

\section{Notation}

$\Lambda$ denotes an even lattice with quadratic form $q$. Often we take $\Lambda = \mathbb{Z}^n$, with $q(v) = \frac{1}{2}v^T S v$ for a Gram matrix $S$ (a symmetric integral matrix with even diagonal). The signature of $\Lambda$ is $(b^+, b^-)$ and its dimension is $e = b^+ + b^-$. The dual lattice is $\Lambda'$. The natural basis of the group ring $\mathbb{C}[\Lambda'/\Lambda]$ is denoted $\mathfrak{e}_{\gamma}$, $\gamma \in \Lambda'/\Lambda.$ Angular brackets $\langle -,- \rangle$ denote the scalar product on $\mathbb{C}[\Lambda'/\Lambda]$ making $\mathfrak{e}_{\gamma}, \, \gamma \in \Lambda'/\Lambda$ an orthonormal basis. \\

$\mathcal{H}$ denotes the Heisenberg group; $\tilde \Gamma$ denotes the metaplectic group; and $\mathcal{J}$ denotes the meta-Jacobi group. $\sigma_{\beta}$ is the Schr\"odinger representation; $\rho$ is the Weil representation; and $\rho_{\beta}$ is a representation of $\mathcal{J}$ that arises as a semidirect product of $\sigma_{\beta}$ and $\rho$. The representations $\sigma_{\beta}^*$, $\rho^*$ and $\rho_{\beta}^*$ are unitary duals of $\sigma_{\beta}, \rho, \rho_{\beta}$. The subgroup $\mathcal{J}_{\infty}$ is the stabilizer of $\mathfrak{e}_0$ under any representation $\rho_{\beta}^*$. The elements $$S = \Big( \begin{pmatrix} 0 & -1 \\ 1 & 0 \end{pmatrix}, \sqrt{\tau} \Big), \; T = \Big( \begin{pmatrix} 1 & 1 \\ 0 & 1 \end{pmatrix}, 1 \Big), \; Z = \Big( \begin{pmatrix} -1 & 0 \\ 0 & -1 \end{pmatrix}, i \Big)$$ are given special names. \\

$E_k = E_{k,0}$ denotes the Eisenstein series (as in \cite{BK}, but normalized to have constant coefficient $1$); more generally, $E_{k,\beta}$ denotes the Eisenstein series with constant term $\frac{\mathfrak{e}_{\beta} + \mathfrak{e}_{-\beta}}{2}.$ With three arguments in the subscript, $E_{k,m,\beta}$ denotes the Jacobi Eisenstein series of weight $k$ and index $m$ for the representation $\rho_{\beta}$. $P_{k,m,\beta}$ denotes the Poincar\'e series of weight $k$ that extracts the coefficient of $q^m \mathfrak{e}_{\beta}$ from cusp forms. Finally, $Q_{k,m,\beta}$ denotes the Poincar\'e square series. Round brackets $(-,-)$ denote the Petersson scalar product of cusp forms. The symbols $|_{k,\rho^*}$ and $|_{k,m,\rho_{\beta}^*}$ denote Petersson slash operators. \\

We will commonly use the abbreviation $\mathbf{e}(x) = e^{2\pi i x}.$ Complex numbers restricted to the upper half-plane are denoted by $\tau = x+iy$; other complex numbers are denoted by $z = u + iv.$

\section{The Weil and Schr\"odinger representations}

The \textbf{metaplectic group} $\tilde \Gamma = Mp_2(\mathbb{Z})$ is the double cover of $SL_2(\mathbb{Z})$ consisting of pairs $(M,\phi)$, where $M$ is a matrix $M = \begin{pmatrix} a & b \\ c & d \end{pmatrix} \in SL_2(\mathbb{Z})$ and $\phi$ is a branch of $\sqrt{c \tau + d}$ on the upper half-plane $$\mathbb{H} = \{\tau = x + iy \in \mathbb{C}: \; y > 0\}.$$ We will typically suppress $\phi$ and denote pairs $(M,\phi)$ by simply giving the matrix $M$. \\

Recall that $\tilde \Gamma$ is presented by the generators $$T = \Big( \begin{pmatrix} 1 & 1 \\ 0 & 1 \end{pmatrix}, 1 \Big), \; \; S = \Big( \begin{pmatrix} 0 & -1 \\ 1 & 0 \end{pmatrix}, \; \sqrt{\tau} \Big),$$ (where $\sqrt{\tau}$ is the ``positive" square root $\mathrm{Im}(\sqrt{\tau}) > 0$, $\tau \in \mathbb{H}$), subject to the relations $S^8 = \mathrm{id}$ and $$S^2 = (ST)^3 = Z = (-I,i).$$

We will also consider the integer \textbf{Heisenberg group} $\mathcal{H}$, which is the set $\mathbb{Z}^3$ with group operation $$(\lambda_1, \mu_1, t_1) \cdot (\lambda_2, \mu_2, t_2) = (\lambda_1 + \lambda_2, \mu_1 + \mu_2, t_1 + t_2 + \lambda_1 \mu_2 - \lambda_2 \mu_1).$$ There is a natural action of $\tilde \Gamma$ on $\mathcal{H}$ (from the right) by $$(\lambda, \mu, t) \cdot \begin{pmatrix} a & b \\ c & d \end{pmatrix} = (a \lambda + c \mu, b \lambda + d \mu, t),$$ and we call the semidirect product $$\mathcal{J} = \mathcal{H} \rtimes \tilde \Gamma$$ by this action the \textbf{meta-Jacobi group}. It can be identified with a subgroup of $Mp_4(\mathbb{Z})$ through the embedding \begin{align*} \mathcal{J} &\rightarrow Mp_4(\mathbb{Z}), \\ \Big(\lambda,\mu,t,\begin{pmatrix} a & b \\ c & d \end{pmatrix}\Big) &\mapsto \begin{pmatrix} a & 0 & b & a \mu - b \lambda \\ \lambda & 1 & \mu & t \\ c & 0 & d & c \mu - d \lambda \\ 0 & 0 & 0 & 1 \end{pmatrix},\end{align*} under which the suppressed square root $\phi(\tau)$ of $c \tau + d$ is sent to $\tilde \phi \Big( \begin{pmatrix} \tau_1 & z \\ z & \tau_2 \end{pmatrix} \Big) = \phi(\tau_1).$

The action of $Mp_4(\mathbb{Z})$ on the Siegel upper half-space $\mathbb{H}_2$ restricts to an action of $\mathcal{J}$ on $\mathbb{H} \times \mathbb{C}$: $$\Big( \lambda,\mu,t, \begin{pmatrix} a & b \\ c & d \end{pmatrix} \Big)  \cdot (\tau,z) = \Big( \frac{a \tau + b}{c \tau + d}, \frac{ \lambda \tau + z + \mu}{c \tau + d} \Big).$$ It can also be shown directly that this defines a group action. \\

Recall that a \textbf{discriminant form} is a finite abelian group $A$ together with a quadratic form $q : A \rightarrow \mathbb{Q}/\mathbb{Z}$, i.e. a function with the properties \\ (i) $q(\lambda x) = \lambda^2 q(x)$ for all $\lambda \in \mathbb{Z}$ and $x \in A$; \\ (ii) $\langle x,y \rangle = q(x+y) - q(x) - q(y)$ is bilinear. \\ The typical example is the discriminant group of an even lattice $\Lambda \subseteq V$ in a finite-dimensional space with bilinear form; ``even'' meaning that $\langle x,x \rangle \in 2\mathbb{Z}$ for all $x \in \Lambda.$ Here, we define the dual lattice $$\Lambda' = \{y \in V: \; \langle x,y \rangle \in \mathbb{Z} \; \mathrm{for} \; \mathrm{all} \; x \in \Lambda\}$$ and take $A$ to be the quotient $A = \Lambda'/ \Lambda$, and set $q(y) = \frac{\langle y,y \rangle}{2}\, \bmod \, 1$ for $y \in A.$ Conversely, every discriminant form arises in this way.

We will review the important representations of $\mathcal{H}$, $\tilde \Gamma$ and $\mathcal{J}$ on the group ring $\mathbb{C}[A]$ of any discriminant form. $\mathbb{C}[A]$ is a complex vector space for which a canonical basis is given by $\mathfrak{e}_{\gamma}$, $\gamma \in A.$ (We will not need the ring structure.) It has a scalar product $$\Big \langle \sum_{\gamma} \lambda_{\gamma} \mathfrak{e}_{\gamma}, \sum_{\gamma} \mu_{\gamma} \mathfrak{e}_{\gamma} \Big \rangle = \sum_{\gamma} \lambda_{\gamma} \overline{\mu_{\gamma}}.$$

\begin{defn} Let $(A,q)$ be a discriminant form and $\beta \in A.$ The \textbf{Schr\"odinger representation} of $\mathcal{H}$ on $\mathbb{C}[A]$ (twisted at $\beta$) is the unitary representation \begin{align*} \sigma_{\beta} : \mathcal{H} &\rightarrow \mathrm{Aut}\, \mathbb{C}[A], \\ \sigma_{\beta}(\lambda,\mu,t) \mathfrak{e}_{\gamma} &= \mathbf{e}\Big( \mu \langle \beta,\gamma \rangle + (t - \lambda \mu) q(\beta) \Big) \mathfrak{e}_{\gamma - \lambda \beta}. \end{align*}
\end{defn}

It is straightforward to check that this actually defines a representation.

\begin{defn} Let $(A,q)$ be a discriminant form. The \textbf{Weil representation} of $\tilde \Gamma$ on $\mathbb{C}[A]$ is the unitary representation $\rho$ defined on the generators $S$ and $T$ by \begin{align*} \rho(T) \mathfrak{e}_{\gamma} &= \mathbf{e}\Big( q(\gamma) \Big) \mathfrak{e}_{\gamma}, \\ \rho(S) \mathfrak{e}_{\gamma} &= \frac{\mathbf{e}((b^- - b^+)/8)}{\sqrt{|\Lambda'/\Lambda|}} \sum_{\beta} \mathbf{e} \Big( -\langle \gamma, \beta \rangle \Big) \mathfrak{e}_{\beta}. \end{align*} Here, $(b^+,b^-)$ is the signature of any lattice with $A$ as its discriminant group; the numbers $b^+, b^-$ are themselves not well-defined, but the difference $b^- - b^+$ mod $8$ depends only on $A$.
\end{defn}

In particular, $$\rho(Z) \mathfrak{e}_{\gamma} = i^{b^- - b^+} \mathfrak{e}_{-\gamma}.$$

Shintani gave in \cite{Shi} an expression for $\rho(M)$, for any $M \in \tilde \Gamma$. We will need this later.

\begin{prop} Let $M = \begin{pmatrix} a & b \\ c & d \end{pmatrix} \in \tilde \Gamma$, and denote by $\rho(M)_{\beta,\gamma}$ the components $$\rho(M)_{\beta,\gamma} = \langle \rho(M) \mathfrak{e}_{\gamma}, \mathfrak{e}_{\beta} \rangle.$$ Suppose that $A$ is the discriminant group of an even lattice $\Lambda$ of signature $(b^+, b^-).$ \\ (i)  If $c = 0$, then $$\rho(M)_{\beta,\gamma} = \sqrt{i}^{(b^- - b^+) (1 - \mathrm{sgn}(d))} \delta_{\beta, a \gamma} \mathbf{e}\Big( ab q(\beta) \Big).$$ (ii) If $c \ne 0$, then $$\rho(M)_{\beta,\gamma} = \frac{\sqrt{i}^{(b^- - b^+)\mathrm{sgn}(c)}}{|c|^{(b^- + b^+)/2} \sqrt{|A|}} \sum_{v \in \Lambda/ c \Lambda} \mathbf{e}\Big( \frac{a q(v+\beta) - \langle \gamma, v + \beta \rangle + d q(\gamma)}{c} \Big).$$
\end{prop}

(Here, $\delta_{\beta, a \gamma} = 1$ if $\beta = a \gamma$ and $0$ otherwise.) Note in particular that $\rho$ factors through a finite-index subgroup of $\tilde \Gamma.$

The following lemma describes the interaction between the Schr\"odinger and Weil representations:

\begin{lem} Let $(A,q)$ be a discriminant form and fix $\beta \in A.$ For any $M \in \tilde \Gamma$ and $\zeta = (\lambda,\mu,t) \in \mathcal{H}$, $$\rho(M)^{-1} \sigma_{\beta}(\zeta) \rho(M) = \sigma_{\beta}(\zeta \cdot M).$$
\end{lem}
\begin{proof} It is enough to verify this when $M$ is one of the standard generators $S$ or $T$. When $M = T$, this is easy to check directly. When $M = S$, $\rho(S)$ is essentially the discrete Fourier transform and this statement is the convolution theorem.
\end{proof}

This implies that $\sigma_{\beta}$ and $\rho$ can be combined to give a unitary representation of the meta-Jacobi group, which we denote by $\rho_{\beta}$: \begin{align*} \rho_{\beta} : \mathcal{J} &\rightarrow \mathrm{Aut} \, \mathbb{C}[A], \\ \rho_{\beta}(\zeta, M) &= \rho(M) \sigma_{\beta}(\zeta) \end{align*} for $M \in \tilde \Gamma$ and $\zeta \in \mathcal{H}.$

We will more often be interested in the dual representations $\sigma_{\beta}^*$, $\rho^*$ and $\rho_{\beta}^*.$ Since all the representations considered here are unitary, we obtain the dual representations essentially by taking complex conjugates everywhere possible.

\section{Modular forms and Jacobi forms}
Fix a lattice $\Lambda.$

\begin{defn} Let $k \in \frac{1}{2}\mathbb{Z}.$ A \textbf{modular form of weight $k$} for the (dual) Weil representation on $\Lambda$ is a holomorphic function $f : \mathbb{H} \rightarrow \mathbb{C}[\Lambda'/\Lambda]$ with the following properties: \\ (i) $f$ transforms under the action of $\tilde \Gamma$ by $$f(M \cdot \tau) = (c \tau + d)^k \rho^*(M) f(\tau), \; M \in \tilde \Gamma,$$ where if $k$ is half-integer then the branch of the square root is prescribed by $M$ as an element of $\tilde \Gamma$. Using the Petersson slash operator, this can be abbreviated as $$f|_{k,\rho^*} M = f \; \; \mathrm{where} \; \; f|_{k,\rho^*}M (\tau) = (c \tau + d)^{-k} \rho^*(M)^{-1} f(M \cdot \tau).$$ (ii) $f$ is holomorphic in $\infty$. This means in the Fourier expansion of $f$, $$f(\tau) = \sum_{\gamma \in \Lambda'/ \Lambda} \sum_{n \in \mathbb{Z} - q(\gamma)} c(n,\gamma) q^n \mathfrak{e}_{\gamma} \; q = e^{2\pi i \tau},$$ all coefficients $c(n,\gamma)$ are zero for $n < 0.$
\end{defn}

(That such a Fourier expansion exists follows from the fact that $f(\tau + 1) = \rho^*(T) f(\tau)$.) \\

The vector space of modular forms will be denoted $M_k(\rho^*)$, and the subspace of cusp forms (those $f$ for which $c(0,\gamma) = 0$ for all $\gamma \in \Lambda'/\Lambda$) is denoted $S_k(\rho^*)$. Both spaces are always finite-dimensional and their dimension (at least for $k \ge 2$) can be calculated with the Riemann-Roch formula. A fast formula for computing this under the assumption that $2k + b^+ - b^- \equiv 0 \, (4)$ was given by Bruinier in \cite{Br2}:

\begin{prop} Define the Gauss sums $$G(a,\Lambda) = \sum_{\gamma \in \Lambda'/ \Lambda} \mathbf{e}\Big( a q(\gamma) \Big), \; \; a \in \mathbb{Z},$$ and define the function $B(x) = x - \frac{\lfloor x \rfloor - \lfloor -x \rfloor}{2}$. Let $d = \#(\Lambda' / \Lambda) / \pm I$ be the number of pairs $\pm \gamma$, $\gamma \in \Lambda'/\Lambda.$ Define $$B_1 = \sum_{\gamma \in \Lambda'/\Lambda} B\Big(q(\gamma)\Big), \; \; B_2 = \sum_{\substack{\gamma \in \Lambda'/\Lambda \\ 2\gamma \in \Lambda}} B\Big(q(\gamma)\Big)$$ and $$\alpha_4 = \#\{\gamma \in \Lambda'/\Lambda: \; q(\gamma) \in \mathbb{Z}\} / \pm I.$$  Then \begin{align*} \mathrm{dim}\, M_k(\rho^*) &= \frac{d (k-1)}{12} \\ &\quad + \frac{1}{4\sqrt{|\Lambda'/\Lambda|}} \mathbf{e}\Big( \frac{2k + b^+ - b^-}{8} \Big) \mathrm{Re}[G(2,\Lambda)] \\ &\quad - \frac{1}{3 \sqrt{3 |\Lambda'/\Lambda|}} \mathrm{Re}\Big[ \mathbf{e}\Big( \frac{4k + 3(b^+ - b^-) - 10}{24} \Big) (G(1,\Lambda) + G(-3,\Lambda)) \Big] \\ &\quad + \frac{\alpha_4 + B_1 + B_2}{2}, \end{align*} and $\mathrm{dim}\, S_k(\rho^*) = \mathrm{dim}\, M_k(\rho^*) - \alpha_4.$
\end{prop}

\begin{defn} (i) The \textbf{Petersson scalar product} on $S_k(\rho^*)$ is $$(f,g) = \int_{\tilde \Gamma \backslash \mathbb{H}} \langle f(\tau), g(\tau) \rangle y^{k-2} \, \mathrm{d}x \, \mathrm{d}y, \; \; \tau = x + iy.$$ Note that $\langle f(\tau), g(\tau) \rangle y^{k-2}\, \mathrm{d}x \, \mathrm{d}y$ is invariant under $\tilde \Gamma.$ \\ (ii) Let $\gamma \in \Lambda'/\Lambda$. The $(n,\gamma)$-th \textbf{Poincar\'e series} (of exponential type) is the cusp form $P_{k,n,\gamma}$ defined by $$(f, P_{k,n,\gamma}) = \frac{\Gamma(k-1)}{(4\pi n)^{k-1}} c(n,\gamma)$$ for any cusp form $f(\tau) = \sum_{\gamma} \sum_n c(n,\gamma) q^n \mathfrak{e}_{\gamma} \in S_k(\rho^*).$ 
\end{defn}

It is clear that the Poincar\'e series span all of $S_k(\rho^*)$.

\begin{prop} For $k \ge 5/2,$ $\gamma \in \Lambda'/\Lambda$ and $n \in \mathbb{Z}-q(\gamma)$ the Poincar\'e series $P_{k,n,\gamma}$ is given by the compactly convergent series $$P_{k,n,\gamma}(\tau) = \sum_{M \in \tilde \Gamma_{\infty} \backslash \tilde \Gamma} \Big( \mathbf{e}(n \tau) \mathfrak{e}_{\gamma} \Big) \Big|_{k,\rho^*} = \frac{1}{2} \sum_{c,d} (c \tau + d)^{-k} \mathbf{e}(n (M \cdot \tau)) \rho^*(M)^{-1} v,$$ where $\tilde \Gamma_{\infty}$ is the subgroup of $\tilde \Gamma$ generated by $T$ and $Z$, and $c,d$ run through all pairs of coprime integers.
\end{prop}
\begin{proof} The series converges compactly since it is majorized by the Eisenstein series; and its definition makes clear that it transforms as a modular form. It is a cusp form because the limit of each term in the series is zero as $y \rightarrow \infty$. To show that it satisfies the characterization by the Petersson scalar product, define $P_{k,n,\gamma}$ by the series above for now; then, for any $f(\tau) = \sum_{\gamma} \sum_{j \in \mathbb{Q}} c(n,\gamma) q^j \mathfrak{e}_{\gamma} \in S_k(\rho^*)$, using the fact that $f|_{k,\rho^*} M = f$ for any $M \in \tilde \Gamma,$ \begin{align*} (f, P_{k,n,\gamma}) &= \int_{\tilde \Gamma \backslash \mathbb{H}} \sum_{M \in \tilde \Gamma_{\infty} \backslash \Gamma} \Big \langle f \Big|_{k,\rho^*} M (\tau), \mathbf{e}(n \tau) \mathfrak{e}_{\gamma} \Big|_{k,\rho^*} M (\tau) \Big \rangle y^{k-2} \, \mathrm{d}x \, \mathrm{d}y \\ &= \int_{-1/2}^{1/2} \int_0^{\infty} \sum_{j \in \mathbb{Q}} \sum_{\beta \in \Lambda'/\Lambda} \langle \mathbf{e}(j \tau) c(j,\beta) \mathfrak{e}_{\beta}, \mathbf{e}(n \tau) \mathfrak{e}_{\gamma} \rangle y^{k-2} \, \mathrm{d}y \, \mathrm{d}x \\ &= \sum_{j \in \mathbb{Z} - q(\gamma)} \Big[ c(j,\gamma) \int_{-1/2}^{1/2} \mathbf{e}((j-n)x) \, \mathrm{d}x \cdot \int_0^{\infty} \mathbf{e}((j+n)y) y^{k-2} \, \mathrm{d}y \Big] \\ &= c(n,\gamma) \int_0^{\infty} e^{-4\pi ny} y^{k-2} \, \mathrm{d}y \\ &= \frac{\Gamma(k-1)}{(4\pi n)^{k-1}} c(n,\gamma). \qedhere \end{align*} 
\end{proof}

Now we define Poincar\'e square series:

\begin{defn} The \textbf{Poincar\'e square series} $Q_{k,m,\beta}$ is the series $$Q_{k,m,\beta} = \sum_{\lambda \in \mathbb{Z}} P_{k,\lambda^2 m, \lambda \beta}.$$
\end{defn}

Here, we set $P_{k.0,0}$ to be the Eisenstein series $E_{k,0}.$ In other words, $Q_{k,m,\beta}$ is the unique modular form such that $Q_{k,m,\beta} - E_{k,0}$ is a cusp form and $$(f, Q_{k,m,\beta} ) = \frac{2 \cdot \Gamma(k-1)}{(4m\pi)^{k-1}} \sum_{\lambda =1}^{\infty} \frac{c(\lambda^2 m, \lambda \beta)}{\lambda^{2k - 2}}$$ for all cusp forms $f(\tau) = \sum_{\gamma,n} c(n,\gamma) q^n \mathfrak{e}_{\gamma}.$

The name ``Poincar\'e square series'' appears to be due to Ziegler in \cite{Z}, where he refers to a scalar-valued Siegel modular form with an analogous definition by that name.

\begin{rem} The components of any cusp form $f = \sum_{n,\gamma} c(n,\gamma) \mathfrak{e}_{\gamma}$ can be considered as scalar-valued modular forms of higher level. Although the Ramanujan-Petersson conjecture is still open in half-integer weight, nontrivial bounds on the growth of $c(n,\gamma)$ are known. For example, Bykovskii (\cite{By}) gives the bound $c(n,\gamma) = O(n^{k/2 - 5/16 + \varepsilon})$ for all $n$ and any $\varepsilon > 0$. This implies that the series $$\sum_{\lambda \ne 0} ( f, P_{k,\lambda^2 m, \lambda \beta} )= \sum_{\lambda \ne 0} \frac{\Gamma(k-1)}{(4\pi \lambda^2 m)^{k-1}} c(\lambda^2 m, \lambda \beta)$$ converges for $k \ge 5/2$. Since $S_k(\rho^*)$ is finite-dimensional, the weak convergence of $\sum_{\lambda \ne 0} P_{k,\lambda^2 m, \lambda \beta}$ actually implies its uniform convergence on compact subsets of $\mathbb{H}.$ On the other hand, the estimate \begin{align*} \sum_{\lambda \in \mathbb{Z}} \Big| \mathbf{e}\Big( m \lambda^2 \frac{a \tau + b}{c \tau + d} \Big) \Big| &= \sum_{\lambda \in \mathbb{Z}} e^{-2\pi m \lambda^2 \frac{y}{|c \tau + d|^2}} \\ &\approx \int_{-\infty}^{\infty} e^{-2\pi m t^2 \frac{y}{|c \tau + d|^2}} \, \mathrm{d}t \\ &= \frac{|c \tau + d|}{\sqrt{2my}}, \; \; y = \mathrm{Im}(\tau) \end{align*} implies that as a triple series, $$Q_{k,m,\beta}(\tau) = \frac{1}{2} \sum_{\lambda \in \mathbb{Z}} \sum_{\mathrm{gcd}(c,d) = 1} (c \tau + d)^{-k} \mathbf{e}\Big( m \lambda^2 \frac{a \tau + b}{ c \tau + d} \Big) \, \rho^*(M)^{-1} \mathfrak{e}_{\lambda \beta}$$ converges absolutely only when $k > 3.$
\end{rem}

\begin{prop} The span of all Poincar\'e square series $Q_{k,m,\beta}$, $m \in \mathbb{N}$, $\beta \in \Lambda'/\Lambda$ contains all of $S_k(\rho^*)$.
\end{prop}
\begin{proof} Since $\mathrm{Span}(Q_{k,m,\beta})$ is finite-dimensional, it is enough to find all Poincar\'e series as weakly convergent infinite linear combinations of $Q_{k,m,\beta}.$ M\"obius inversion implies the formal identity 
$$P_{k,m,\beta} = \frac{1}{2}\Big( P_{k,m,\beta} + P_{k,m,-\beta}\Big) = \frac{1}{2} \sum_{d \in \mathbb{N}} \mu(d) \Big[ Q_{k,d^2 m, d \beta} - E_{k,0} \Big].$$ The series on the right converges (weakly) in $S_k(\rho^*)$ because we can bound $$\Big|( f, Q_{k,d^2 m, d \beta} ) \Big| \le \sum_{\lambda \in \mathbb{Z}} \frac{\Gamma(k-1)}{(4 \pi \lambda^2 d^2 m)^{k-1}} \Big| c(\lambda^2 d^2 m, \lambda d \beta) \Big| \le C \cdot d^{-9/8 + \varepsilon}$$ for an appropriate constant $C$ and all cusp forms $f(\tau) = \sum_{\gamma} \sum_n c(n,\gamma) q^n \mathfrak{e}_{\gamma}$, where we again use the bound $c(n,\gamma) = O(n^{k/2 - 5/16 + \varepsilon}).$
\end{proof}

Finally, we will need to define Jacobi forms. We will consider Jacobi forms for the representation $\rho_{\beta}^*$ defined in section 3. The book \cite{EZ} remains the standard reference for (scalar-valued) Jacobi forms, and much of the following work is based on the calculations there. \\

\begin{defn} A \textbf{Jacobi form} for $\rho_{\beta}^*$ of weight $k$ and index $m$ is a holomorphic function $\Phi : \mathbb{H} \times \mathbb{C} \rightarrow \mathbb{C}[\Lambda'/\Lambda]$ with the following properties: \\ (i) For any $M = \begin{pmatrix} a & b \\ c & d \end{pmatrix} \in \tilde \Gamma$, $$\Phi\Big( \frac{ a \tau + b}{c \tau + d} , \frac{z}{c \tau + d} \Big) = (c \tau + d)^k \mathbf{e}\Big( \frac{mcz^2}{c \tau + d} \Big) \cdot \rho^*(M) \Phi(\tau,z);$$ (ii) For any $\zeta = (\lambda,\mu,t) \in \mathcal{H}$, $$\Phi(\tau, z+ \lambda \tau + \mu) = \mathbf{e}\Big( -m\lambda^2 \tau - 2m \lambda z - m(\lambda \mu + t) \Big) \cdot \sigma_{\beta}^*(\zeta)\Phi(\tau,z);$$ (iii) If we write out the Fourier series of $\Phi$ as $$\Phi(\tau,z) = \sum_{\gamma \in \Lambda'/\Lambda} \sum_{n,r \in \mathbb{Q}} c(n,r,\gamma) q^n \zeta^r \mathfrak{e}_{\gamma}, \; \; q = e^{2\pi i \tau}, \; \zeta = e^{2\pi i z},$$ then $c(n,r,\gamma) = 0$ whenever $n < r^2 / 4m$.
\end{defn}

We define a Petersson slash operator in this setting as follows: for $M \in \tilde \Gamma$ and $\zeta \in \mathcal{H}$, \begin{align*} &\quad \Phi \Big|_{k,m,\rho_{\beta}^*} (\zeta,M) (\tau,z) \\ &= (c \tau + d)^{-k} \mathbf{e}\Big( m \lambda^2 \tau + 2m \lambda z  + m(\lambda \mu + t) - \frac{cm (z + \lambda \tau + \mu)^2}{c \tau + d} \Big) \cdot \rho_{\beta}^*(\zeta, M)^{-1} \Big[ \Phi\Big( \frac{a \tau + b}{c \tau + d}, \frac{z + \lambda \tau + \mu}{c \tau + d} \Big) \Big].\end{align*} Then conditions (i),(ii) of being a Jacobi form can be summarized as $$\Phi \Big|_{k,m,\rho_{\beta}^*} (\zeta, M) = \Phi$$ for all $(\zeta,M) \in \mathcal{J}.$

\begin{rem} We will consider some basic consequences of the transformation law under $\mathcal{J}$ for a Jacobi form $\Phi(\tau,z) = \sum_{\gamma,n,r} c(n,r,\gamma) q^n \zeta^r \mathfrak{e}_{\gamma}.$ First, letting $\zeta = (0,0,1) \in \mathcal{H}$, we see that $$\Phi = \mathbf{e}\Big( -m - q(\beta) \Big) \Phi$$ so there are no nonzero Jacobi forms unless $m \in \mathbb{Z} - q(\beta).$  Also, $$\sum_{\gamma} \sum_{n,r} c(n,r,\gamma) \mathbf{e}(n) q^n \zeta^r \mathfrak{e}_{\gamma} = \Phi(\tau + 1,z) = \rho^*(T) \Phi(\tau) = \sum_{\gamma} \sum_{n,r} c(n,r,\gamma) \mathbf{e}(-q(\gamma)) q^n \zeta^r \mathfrak{e}_{\gamma}$$ implies that $c(n,r,\gamma) = 0$ unless $n \in \mathbb{Z} - q(\gamma).$ Similarly, $$\sum_{\gamma} \sum_{n,r} c(n,r,\gamma) \mathbf{e}(r) q^n \zeta^r \mathfrak{e}_{\gamma} = \Phi( \tau,z+1) = \sigma_{\beta}^*(0,1,0) \Phi(\tau,z) = \sum_{\gamma} \sum_{n,r} c(n,r,\gamma) \mathbf{e}\Big(-\langle \beta,\gamma \rangle \Big) q^n \zeta^r \mathfrak{e}_{\gamma}$$ implies that $c(n,r,\gamma) = 0$ unless $r \in \mathbb{Z} - \langle \gamma, \beta \rangle.$ The transformation under $Z$ implies $$\sum_{n,r,\gamma} c(n,r,\gamma) q^n \zeta^{-r}\mathfrak{e}_{\gamma} = \Phi(\tau,-z) = (-1)^k \rho^*(Z) \Phi(\tau,z) = i^{2k + b^+ - b^-} \sum_{n,r,\gamma} c(n,r,\gamma) q^n \zeta^r \mathfrak{e}_{-\gamma},$$ so there are no nonzero Jacobi forms unless $2k + b^+ - b^- \in 2 \mathbb{Z}$. (We will always make the assumption $$2k + b^+ - b^- \in 4 \mathbb{Z},$$ since the $\mathfrak{e}_0$-component of any Jacobi form will otherwise vanish identically. In this case $c(n,r,\gamma) = c(n,-r,-\gamma)$ for all $n,r,\gamma.$) Finally, we remark that the transformation under $\zeta = (\lambda,0,0)$ implies \begin{align*} \sum_{n,r,\gamma} c(n,r,\gamma) q^{n + r \lambda} \zeta^r \mathfrak{e}_{\gamma} &= \Phi(\tau,z + \lambda \tau) \\ &= q^{-m \lambda^2} \zeta^{-2m \lambda} \sigma_{\beta}^*(\lambda,0,0) \Phi(\tau,z) \\ &= \sum_{n,r,\gamma} c(n,r,\gamma) q^{n - m \lambda^2} \zeta^{r - 2m \lambda} \mathfrak{e}_{\gamma - \lambda \beta} \end{align*} and therefore $c(n, r,\gamma) = c(n + r \lambda + m \lambda^2, r + 2m \lambda, \gamma + \lambda \beta)$ for all $\lambda \in \mathbb{Z}$.
\end{rem}

\section{The Jacobi Eisenstein series}

Fix a lattice $\Lambda.$ Let $\mathcal{J}_{\infty}$ denote the subgroup of $\mathcal{J}$ that fixes the constant function $\mathfrak{e}_0$ under the action $|_{k,m,\rho^*_{\beta}}$. This is independent of $\beta$ and it is the group generated by $T,Z \in \tilde \Gamma$ and the elements of the form $(0,\mu,t) \in \mathcal{H}$ in the Heisenberg group. \\

\begin{defn} The \textbf{Jacobi Eisenstein series} twisted at $\beta \in \Lambda'$ of weight $k$ and index $m \in \mathbb{Z}-q(\beta)$ is $$E_{k,m,\beta}(\tau,z) = \sum_{(M,\zeta) \in \mathcal{J}_{\infty} \backslash \mathcal{J}} \mathfrak{e}_0 \Big|_{k,m,\rho^*_{\beta}} (M,\zeta) (\tau,z).$$ It is clear that this is a Jacobi form of weight $k$ and index $m$ for the representation $\rho_{\beta}^*$. More explicitly, we can write it in the form $$E_{k,m,\beta}(\tau,z) = \frac{1}{2} \sum_{c,d} (c \tau + d)^{-k} \sum_{\lambda \in \mathbb{Z}} \mathbf{e} \Big( m \lambda^2 (M \cdot \tau) + \frac{2m \lambda z}{c \tau + d} - \frac{cmz^2}{c \tau + d} \Big) \rho^*(M)^{-1} \sigma_{\beta}^*(\lambda,0,0)^{-1} \mathfrak{e}_0.$$
\end{defn}

\begin{rem} This series converges absolutely when $k > 3$. In that case the zero-value $E_{k,m,\beta}(\tau,0)$ is the Poincar\'e square series $Q_{k,m,\beta}(\tau)$, as one can see by swapping the order of the sum over $(c,d)$ and the sum over $\lambda.$
\end{rem}

$E_{k,m,\beta}$ has a Fourier expansion of the form $$E_{k,m,\beta}(\tau,z) = \sum_{\gamma \in \Lambda'/\Lambda} \sum_{n \in \mathbb{Z} - q(\gamma)} \sum_{r \in \mathbb{Z} - \langle \gamma,\beta \rangle} c(n,r,\gamma) q^n \zeta^r \mathfrak{e}_{\gamma}.$$ We will calculate its coefficients. The contribution from $c = 0$ and $d = \pm 1$ is $$\sum_{\lambda \in \mathbb{Z}} \mathbf{e} \Big( m \lambda^2 \tau + 2m \lambda z \Big) \mathfrak{e}_{\lambda \beta}.$$ We denote the contribution from all other terms by $c'(n,r,\gamma)$; so $$E_{k,m,\beta}(\tau,z) = \sum_{\lambda \in \mathbb{Z}} \mathbf{e}\Big(m \lambda^2 \tau + 2 m \lambda z \Big) \mathfrak{e}_{\lambda \beta} + \sum_{\gamma \in \Lambda'/\Lambda} \sum_{n \in \mathbb{Z} - q(\gamma)} \sum_{r \in \mathbb{Z} - \langle \gamma,\beta \rangle} c'(n,r,\gamma) q^n \zeta^r \mathfrak{e}_{\gamma}.$$ Write $\tau = x+iy$ and $z = u + iv$. Then $c'(n,r,\gamma)$ is given by the integral \begin{align*} &\quad c'(n,r,\gamma) \\ &= \frac{1}{2} \int_0^1 \int_0^1 \sum_{c \ne 0} \sum_{\lambda} (c \tau + d)^{-k} \mathbf{e} \Big( m \lambda^2 (M \cdot \tau) + \frac{2m \lambda z}{c \tau + d} - \frac{cmz^2}{c \tau + d} \Big) \mathbf{e} (-n \tau - rz) \langle  \rho^*(M)^{-1} \sigma_{\beta}^*(\lambda,0,0)^{-1} \mathfrak{e}_0, \mathfrak{e}_{\gamma} \rangle \,  \mathrm{d}x \, \mathrm{d}u \\ &= \frac{1}{2} \sum_{c \ne 0} \sum_{d \, (c)^*} \sum_{\lambda} \rho(M)_{\lambda \beta, \gamma} \int_{-\infty}^{\infty} \int_0^1 (c \tau + d)^{-k} \mathbf{e} \Big( -n \tau - rz + m \lambda^2 (M \cdot \tau) + \frac{2m \lambda z}{c \tau +d} - \frac{cmz^2}{c \tau + d} \Big) \, \mathrm{d}u \, \mathrm{d}x. \end{align*} Here, the notation $\sum_{d \, (c)^*}$ implies that the sum is taken over representatives of $\Big( \mathbb{Z}/c \mathbb{Z} \Big)^{\times}.$ The double integral simplifies to \begin{align*} &\quad \int_{-\infty}^{\infty} \int_0^1 (c \tau + d)^{-k} \mathbf{e} \Big( -n \tau - rz + m \lambda^2 (M \cdot \tau) + \frac{2m \lambda z}{c \tau +d} - \frac{cmz^2}{c \tau + d} \Big) \, \mathrm{d}u \, \mathrm{d}x \\ &=c^{-k} \mathbf{e} \Big( \frac{am \lambda^2 + nd}{c} \Big) \int_{-\infty}^{\infty} \int_0^1 \mathbf{e} \Big( -n \tau - rz - m(cz - \lambda)^2 / (c^2 \tau) \Big) \, \mathrm{d}u \, \mathrm{d}x\end{align*} by substituting $\tau - d/c$ into $\tau.$ \\

The inner integral over $u$ is easiest to evaluate within the sum over $\lambda$. Namely, \begin{align*} &\quad \sum_{\lambda \in \mathbb{Z}} \rho(M)_{\lambda \beta, \gamma} \mathbf{e} \Big( \frac{am \lambda^2}{c} \Big) \int_0^1 \mathbf{e} \Big( - rz - m \frac{(cz - \lambda)^2}{c^2 \tau} \Big) \, \mathrm{d}u \\ &= \sum_{\lambda \in \mathbb{Z}} \rho(M)_{\beta \lambda, \gamma} \mathbf{e} \Big( \frac{am \lambda^2 - r \lambda}{c} \Big) \int_{-\lambda / c}^{1 - \lambda / c} \mathbf{e} \Big( -rz - mz^2 / \tau \Big) \, \mathrm{d}u \end{align*} after substituting $z + \lambda / c$ into $z$. Note that \begin{align*} &\quad \rho(M)_{\lambda \beta, \gamma} \mathbf{e} \Big( \frac{am \lambda^2 - r\lambda}{c} \Big) \\ &= \frac{\sqrt{i}^{(b^- - b^+) \mathrm{sgn}(c)}}{|c|^{(b^- + b^+) / 2} \sqrt{|\Lambda'/\Lambda|}} \sum_{v \in \Lambda / c \Lambda} \mathbf{e} \Big( \frac{a q(v + \lambda \beta) - \langle v + \lambda \beta, \gamma \rangle + d q(\gamma) + a m \lambda^2 - r \lambda}{c} \Big) \\ &= \frac{\sqrt{i}^{(b^- - b^+) \mathrm{sgn}(c)}}{|c|^{(b^- + b^+) / 2} \sqrt{|\Lambda'/\Lambda|}}  \sum_{v \in \Lambda / c \Lambda} \mathbf{e} \Big( \frac{a \lambda^2 [m + q(\beta)] + \lambda [a \langle v, \beta \rangle - \langle \beta, \gamma \rangle - r] + a q(v) - \langle v, \gamma \rangle + d q(\gamma)}{c} \Big) \end{align*} depends only on the remainder of $\lambda$ mod $c$, because $m + q(\beta)$ and $r + \langle \beta, \gamma \rangle$ are integers. Continuing, we see that \begin{align*}&\quad \sum_{\lambda \in \mathbb{Z}} \rho(M)_{\beta \lambda, \gamma} \mathbf{e} \Big( \frac{am \lambda^2 - r \lambda}{c} \Big) \int_{-\lambda /c}^{1 - \lambda / c} \mathbf{e} \Big( -rz - mz^2 / \tau \Big) \, \mathrm{d}u \\ &= \frac{\sqrt{i}^{(b^- - b^+) \mathrm{sgn}(c)}}{|c|^{(b^- + b^+) / 2} \sqrt{|\Lambda'/\Lambda|}}  \sum_{\substack{v \in \Lambda / c \Lambda \\ \lambda \in \mathbb{Z}/c \mathbb{Z}}}  \mathbf{e} \Big( \frac{a \lambda^2 [m + q(\beta)] + \lambda [a \langle v, \beta \rangle - \langle \beta, \gamma \rangle - r] + a q(v) - \langle v, \gamma \rangle + d q(\gamma)}{c} \Big) \times \\ &\quad\quad\quad \times \int_{-\infty}^{\infty} \mathbf{e} \Big( -rz - mz^2 / \tau \Big) \, \mathrm{d}u. \end{align*}

The Gaussian integral is well-known: $$\int_{-\infty}^{\infty} \mathbf{e} \Big( -rz - mz^2 / \tau \Big) \, \mathrm{d}u = \mathbf{e}\Big( r^2 \tau/4m \Big) \sqrt{\tau / 2im}.$$ We are left with \begin{align*} c'(n,r,\gamma) = \frac{1}{2 \sqrt{2im}} \sum_{c \ne 0} \frac{\sqrt{i}^{(b^- - b^+) \mathrm{sgn}(c)}}{|c|^{(b^- + b^+) / 2} \sqrt{|\Lambda'/\Lambda|}}  c^{-k} K_c(\beta,m,\gamma,n,r) \int_{-\infty}^{\infty} \tau^{1/2 - k} \mathbf{e}\Big(\tau (r^2 / 4m -n) \Big) \, \mathrm{d}x, \end{align*} where $K_c(\beta,m,\gamma,n,r)$ is a Kloosterman sum: \begin{align*} &\quad K_c(\beta,m,\gamma,n,r) \\ &= \sum_{d \, (c)^*} \sum_{\substack{v \in \Lambda / c \Lambda \\ \lambda \in \mathbb{Z}/c \mathbb{Z}}}  \mathbf{e} \Big( \frac{a \lambda^2 [m + q(\beta)] + \lambda [a \langle v, \beta \rangle - \langle \beta, \gamma \rangle - r] + a q(v) - \langle v, \gamma \rangle + d q(\gamma) + dn}{c} \Big) \\ &= \sum_{\substack{v \in \Lambda / c \Lambda \\ \lambda \in \mathbb{Z}/c \mathbb{Z}}} \sum_{d \, (c)^*} \mathbf{e} \Big( \frac{d}{c} \Big[ \lambda^2 (m + q(\beta)) + \lambda (\langle v,\beta \rangle - \langle \gamma,\beta \rangle - r) + q(v) - \langle v, \gamma \rangle + q(\gamma) +n \Big] \Big) \\ &= \sum_{\substack{v \in \Lambda / c \Lambda \\ \lambda \in \mathbb{Z}/c \mathbb{Z}}} \sum_{d \, (c)^*} \mathbf{e} \Big( \frac{d}{c}\Big[q(v + \lambda \beta - \gamma) + m \lambda^2 - r \lambda + n\Big] \Big). \end{align*} (In the second equality we have replaced $v$ and $\lambda$ by $d \cdot v$ and $d \cdot  \lambda$.)

The integral $\int_{-\infty}^{\infty} \tau^{1/2 - k} \mathbf{e}\Big( \tau (r^2 / 4m - n) \Big) \, \mathrm{d}x$ is $0$ when $r^2 / 4m - n \ge 0$, since the integral is independent of $y = \mathrm{Im}(\tau)$ and tends to $0$ as $y \rightarrow \infty.$ When $r^2 / 4m - n < 0$, we deform the contour to a keyhole and use Hankel's integral $$\frac{1}{\Gamma(s)} = \frac{1}{2\pi i} \oint_{\gamma} e^\tau \tau^{-s} \, \mathrm{d}\tau$$ to conclude that $$\int_{-\infty}^{\infty} \tau^{1/2 - k} \mathbf{e}\Big( \tau (r^2 / 4m - n) \Big) \, \mathrm{d}x = \frac{2\pi i \cdot  (2\pi i (r^2 / 4m - n))^{k - 3/2}}{\Gamma(k-1/2)}$$ and therefore 
\begin{align*} c'(n,r,\gamma) &= \frac{(2\pi i)^{k-1/2} (r^2 / 4m - n)^{k-3/2}}{2 \cdot  \Gamma(k-1/2) \sqrt{2im |\Lambda'/\Lambda|}} \sum_{c \ne 0} \frac{\sqrt{i}^{(b^- - b^+) \mathrm{sgn}(c)}}{|c|^{(b^- + b^+) / 2}} c^{-k} K_c(\beta,m,\gamma,n,r) \\ &= \frac{(-i)^k \pi^{k-1/2} (4mn - r^2)^{k-3/2}}{2^{k-3} m^{k-1} \Gamma(k-1/2) \sqrt{|\Lambda'/\Lambda|}} \sum_{c \ne 0} \frac{\sqrt{i}^{(b^- - b^+) \mathrm{sgn}(c)}}{|c|^{(b^- + b^+) / 2}} c^{-k} K_c(\beta,m,\gamma,n,r). \end{align*}

We can use $$\sqrt{i}^{(b^- - b^+) \mathrm{sgn}(c)} \mathrm{sgn}(c)^k (-i)^k = (-1)^{(2k - b^- + b^+) / 4}$$ and the fact that $K_c(\beta,m,\gamma,n,r) = K_{-c}(\beta,m,\gamma,n,r)$ to write this as $$c'(n,r,\gamma) = \frac{(-1)^{(2k - b^- + b^+) / 4} \pi^{k-1/2} (4mn - r^2)^{k-3/2}}{2^{k-2} m^{k-1} \Gamma(k-1/2) \sqrt{|\Lambda'/\Lambda|}} \sum_{c = 1}^{\infty} c^{-k - e/2} K_c(\beta,m,\gamma,n,r).$$

\begin{rem} Using the evaluation of the Ramanujan sum, $$\sum_{d \, (c)^*} \mathbf{e} \Big( \frac{d}{c}N\Big) = \sum_{a | (c,N)} \mu(c/a) a,$$ where $\mu$ is the M\"obius function, it follows that \begin{align*} &\quad K_c(\beta,m,\gamma,n,r) \\ &= \sum_{a | c} \mu(c / a) a \cdot \#\Big\{ (v,\lambda) \in (\Lambda \oplus \mathbb{Z}) / (c): \; q(v + \lambda \beta - \gamma) + m \lambda^2 - r \lambda + n = 0 \, (c) \} \\ &=  \sum_{a | c} \mu(c/a)a (c/a)^{e+1} \cdot \#\Big\{ (v,\lambda) \in (\Lambda \oplus \mathbb{Z}) / (a): \;  q(v + \lambda \beta - \gamma) + m \lambda^2 - r \lambda + n = 0 \, (c) \Big\} \\ &= c^{e+1} \sum_{a|c} \mu(c/a) a^{-e} \mathbf{N}(a), \end{align*} where we define $$\mathbf{N}(a) = \#\Big\{ (v,\lambda) \in (\Lambda \oplus \mathbb{Z}) / (a): \;  q(v + \lambda \beta - \gamma) + m \lambda^2 - r \lambda + n \equiv 0 \, (a) \Big\}$$ and we use the fact that this congruence depends only on the remainder of $v$ and $\lambda$ mod $a$ (rather than $c$).
\end{rem}

\begin{rem} If we identify $\Lambda = \mathbb{Z}^n$ and write $q$ as $q(v) = \frac{1}{2} v^T S v$ with a symmetric integer matrix $S$ with even diagonal (its Gram matrix), then we can rewrite \begin{align*} &\quad \lambda^2 m + q(v + \lambda \beta - \gamma) - r \lambda + n \\ &= \frac{1}{2} (\tilde v - \tilde \gamma)^T \begin{pmatrix} S & S \beta \\ (S \beta)^T & 2(m + q(\beta)) \end{pmatrix} (\tilde v - \tilde \gamma) + \tilde n \end{align*} with $\tilde v = (v,\lambda)$ and $\tilde \gamma = (\gamma, -\frac{r}{2(m + q(\beta))})$ and $\tilde n = n + \frac{r}{2(m + q(\beta))} \langle \gamma, \beta \rangle - \frac{r^2}{4 (m + q(\beta))}.$ Therefore, $\mathbf{N}(a)$ equals the representation number $N_{\tilde \gamma, \tilde n}(a)$ in the notation of \cite{BK}. The analysis there does not seem to apply to this situation because $\tilde \gamma$ has no reason to be in the dual lattice of this larger quadratic form, and because $\tilde n$ can be negative or even zero. \\

In the particular case $\beta = 0$, the coefficient $c(n,r,\gamma)$ does in fact occur as the coefficient of $$(\tilde n,\tilde \gamma) = (n-r^2 / 4m, (\gamma,r/2m))$$ in the Eisenstein series $E_{k,0}$ attached to the lattice with Gram matrix $\begin{pmatrix} S & 0 \\ 0 & 2m \end{pmatrix}.$ This can be seen as a case of the theta decomposition, which gives more generally an isomorphism between Jacobi forms for a trivial action of the Heisenberg group and vector-valued modular forms, and identifies Jacobi Eisenstein series with vector-valued Eisenstein series.
\end{rem}

\begin{rem} We consider the Dirichlet series $$\tilde L(s) = \sum_{c=1}^{\infty} c^{-s} K_c(\beta,m,\gamma,n,r).$$ Since $K_c$ is $c^{e+1}$ times the convolution of $\mu(a)$ and $a^{-e} \mathbf{N}(a)$, it follows formally that $$\tilde L(s+e+1) = \zeta(s)^{-1} L(s+e)$$ where we have defined $$L(s) = \sum_{c=1}^{\infty} c^{-s} \mathbf{N}(c).$$ Since $\mathbf{N}(a)$ is multiplicative (for coprime $a_1,a_2$, a pair $(v,\lambda)$ solves the congruence modulo $a_1 a_2$ if and only if it does so modulo both $a_1$ and $a_2$), $L(s)$ can be written as an Euler product $$L(s) = \prod_{p \, \mathrm{prime}} L_p(s) \; \; \mathrm{with} \; \; L_p(s) = \sum_{\nu=0}^{\infty} \mathbf{N}(p^{\nu}) p^{-\nu s}.$$ The functions $L_p(s)$ are always rational functions in $p^{-s}$ and in particular they have a meromorphic extension to $\mathbb{C}$; and it follows that $c'(n,r,\gamma)$ is the value of the analytic continuation of $$\frac{(-1)^{(2k - b^- + b^+) / 4} \pi^{k-1/2} (4mn - r^2)^{k-3/2}}{2^{k-2} m^{k-1} \Gamma(k-1/2) \zeta(s - e) \sqrt{|\Lambda'/\Lambda|}} \prod_{p \, \mathrm{prime}} L_p(s)$$ at $s = k + e/2 - 1.$
\end{rem}

\section{Evaluation of $L_p(s)$}

In this section we review the calculation of Igusa zeta functions of quadratic polynomials due to Cowan, Katz and White in \cite{CKW} and apply it to calculate the Euler factors $L_p(k + e/2 - 1).$ \\

\begin{defn} Let $f \in \mathbb{Z}_p[X_1,...,X_e]$ be a polynomial of $e$ variables. The \textbf{Igusa zeta function} of $f$ at a prime $p$ is the $p$-adic integral $$\zeta_{Ig}(f;p;s) = \int_{\mathbb{Z}_p^e} |f(x)|^s \, \mathrm{d}x, \; \; s \in \mathbb{C}.$$ In other words, $$\zeta_{Ig}(f;p;s) = \sum_{\nu=0}^{\infty}\mathrm{Vol}\Big(\{ x \in \mathbb{Z}_p^e: \; |f(x)|_p = p^{-\nu} \} \Big) p^{-\nu s},$$ where $\mathrm{Vol}$ denotes the Haar measure on $\mathbb{Z}_p^e$ normalized such that $\mathrm{Vol}(\mathbb{Z}_p^e) = 1.$
\end{defn}

Igusa proved (\cite{Ig}) that $\zeta_{Ig}(f;p;s)$, which is a priori only a formal power series in $p^{-s}$, is in fact a rational function of $p^{-s}$. In particular, it has a meromorphic continuation to all of $\mathbb{C}$. \\

Our interest in the Igusa zeta function is due to the identity of generating functions $$\frac{1 - p^{-s} \zeta_{Ig}(f;p;s)}{1 - p^{-s}} = \sum_{\nu=0}^{\infty} N_f(p^{\nu}) p^{-\nu (s+e)},$$ where $N_f(p^{\nu})$ denotes the number of solutions $$N_f(p^{\nu}) = \# \Big\{ x \in \mathbb{Z}^e / p^{\nu} \mathbb{Z}^e: \; f(x) \equiv 0 \, \bmod \, p^{\nu} \Big\}.$$ In particular, $$L_p(s) = \frac{1 - p^{-s+e + 1} \zeta_{Ig}(f;p;s-e - 1)}{1 - p^{-s+e + 1}}$$ for the polynomial of $(e+1)$ variables $$f(v,\lambda) = \lambda^2 m + q(v + \lambda \beta - \gamma) - r \lambda + n.$$

The calculation of $\zeta_{Ig}(f;p;s)$ will be stated for quadratic polynomials in the form $$f = \bigoplus_{i \in \mathbb{N}_0} p^i Q_i \oplus L + c,$$ where $Q_i$ are unimodular quadratic forms, $L$ is a linear form involving at most one variable, and $c \in \mathbb{Z}_p$. The notation $\bigoplus$ implies that no two terms in this sum contain any variables in common. To any quadratic polynomial $g$, there exists a polynomial $f$ as above that is ``isospectral'' to $g$ at $p$, in the sense that $N_f(p^{\nu}) = N_g(p^{\nu})$ for all $\nu \in \mathbb{N}_0$. Consult section 4.9 of \cite{CKW} for an algorithm to compute $f$. We will say that polynomials $f$ as above are in \textbf{normal form}. \\

\begin{prop} Let $p$ be an odd prime. Let $f(X) = \bigoplus_{i \in \mathbb{N}_0} p^i Q_i(X) \oplus L(X) + c$ be a $\mathbb{Z}_p$-integral quadratic polynomial in normal form, and fix $\omega \in \mathbb{N}$ such that $Q_i = 0$ for $i > \omega.$ Define $$r_i = \mathrm{rank}(Q_i) \; \; \mathrm{and} \; \; d_i = \mathrm{disc}(Q_i), \; i \in \mathbb{N}_0$$ and $$\mathbf{r}_{(j)} = \sum_{\substack{0 \le i \le j \\ i \equiv j \, (2)}} r_i \; \; \mathrm{and} \; \; \mathbf{d}_{(j)} = \prod_{\substack{0 \le i \le j \\ i \equiv j \, (2)}} d_i, \; j \in \mathbb{N}_0,$$ and also define $$\mathbf{p}_{(j)} = p^{\sum_{0 \le i < j} \mathbf{r}_{(i)}}, \; j \in \mathbb{N}_0.$$ Define the helper functions $I_a(r,d)(s)$ by $$I_a(r,d)(s) = \begin{cases} (1 - p^{-s-r}) \frac{p-1}{p - p^{-s}}: & r \, \mathrm{odd}, \; p | a; \\ \\ \Big[ 1 + p^{-s - (r+1)/2} \Big( \frac{ad (-1)^{(r+1)/2}}{p} \Big) \Big] \frac{p-1}{p - p^{-s}} - p^{-r} - p^{-(r+1)/2} \Big( \frac{ad (-1)^{(r+1)/2}}{p} \Big): & r \, \mathrm{odd}, \; p \nmid a; \\ \\ \Big[ 1 - p^{-r/2} \Big( \frac{(-1)^{r/2} d}{p}\Big) \Big] \cdot \Big[ 1 + p^{-s-r/2} \Big( \frac{(-1)^{r/2} d}{p} \Big) \Big] \frac{p-1}{p-p^{-s}}: & r \, \mathrm{even}, \; p | a; \\ \\ \Big[ 1 - p^{-r/2} \Big( \frac{(-1)^{r/2} d}{p} \Big) \Big] \cdot \Big[ \frac{p-1}{p-p^{-s}} + p^{-r/2} \Big( \frac{(-1)^{r/2} d}{p} \Big) \Big]: & r \, \mathrm{even}, \; p \nmid a, \end{cases}$$ where $\Big( \frac{a}{p} \Big)$ is the quadratic reciprocity symbol on $\mathbb{Z}_p$. Then: \\ (i) If $L = 0$ and $c = 0$, let $r = \sum_{i \in \mathbb{N}_0} r_i$; then $$\zeta_{Ig}(f;p;s) = \sum_{0 \le \nu < \omega - 1} \frac{I_0(\mathbf{r}_{(\nu)},\mathbf{d}_{(\nu)})}{\mathbf{p}_{(\nu)}} p^{-\nu s} + \Big[ \frac{I_0(\mathbf{r}_{(\omega - 1)},\mathbf{d}_{(\omega-1)})}{\mathbf{p}_{(\omega-1)}} p^{-(\omega - 1)s} + \frac{I_0(\mathbf{r}_{(\omega)},\mathbf{d}_{(\omega)})}{\mathbf{p}_{(\omega)}} p^{-\omega s} \Big] \cdot (1 - p^{-2s - r})^{-1}.$$ (ii) If $L(x) = bx$ with $b \ne 0$ and $v_p(c) \ge v_p(b)$, let $\lambda = v_p(b)$; then $$\zeta_{Ig}(f;p;s) = \sum_{0 \le \nu < \lambda} \frac{I_0(\mathbf{r}_{(\nu)},\mathbf{d}_{(\nu)})}{\mathbf{p}_{(\nu)}} p^{-\nu s} + \frac{p^{-\lambda s}}{\mathbf{p}_{(\lambda)}} \cdot \frac{p - 1}{p - p^{-s}}.$$ (iii) If $L = 0$ and $c \ne 0$, or if $L(x) = bx$ with $v_p(b) > v_p(c)$, let $\kappa = v_p(c)$; then $$\zeta_{Ig}(f;p;s) = \sum_{0 \le \nu \le \kappa} \frac{I_{c/p^{\nu}}(\mathbf{r}_{(\nu)}, \mathbf{d}_{(\nu)})}{\mathbf{p}_{(\nu)}} p^{-\nu s} + \frac{1}{\mathbf{p}_{(\kappa + 1)}} p^{-\kappa s}.$$
\end{prop}
\begin{proof} This is theorem 2.1 of \cite{CKW}. We have replaced the variable $t$ there by $p^{-s}.$
\end{proof}

\begin{rem} Since the constant term here is never $0$, we are always in either case (ii) or case (iii). It follows that the only possible pole of $\zeta_{Ig}(f;p;s)$ is at $s = -1$, and therefore the only possible poles of $L_p(s)$ are at $e$ or $e + 1.$ Therefore, the value $k + e/2 - 1$ is not a pole of $L_p$, with the weights $k = e/2 + 1$ or $k = e/2 + 2$ as the only possible exceptions. In fact, $k = e/2 + 1$ can occur as a pole but this is ultimately cancelled out by the prescence of $\zeta(k-e/2-1)$ in the denominator of $c'(n,r,\gamma)$, and $k = e/2 + 2$ never occurs as a pole (as one can show by bounding $\mathbf{N}$). \\

An easy, if unsatisfying, proof that $e/2 + 2$ could not occur as a pole is that the problem can be avoided entirely by appending hyperbolic planes (or other unimodular lattices) to $\Lambda$, which does not change the discriminant group and therefore does not change the coefficients of $E_{k,m,\beta}$, but makes $e$ arbitrarily large.
\end{rem}

\begin{rem} Identify $\Lambda = \mathbb{Z}^n$ and $q(v) = \frac{1}{2} v^T S v$ where $S$ is the Gram matrix. We will use proposition 20 to calculate $$L_p(k+e/2 - 1) = \frac{1 - p^{-k + e/2 + 2} \zeta_{Ig}(f;p;k-e/2 - 2)}{1 - p^{-k + e/2 + 2}}$$ for ``generic" primes $p$ - these are primes $p \ne 2$ at which $$\mathrm{det}(S), \; \; d_{\beta}^2 m, \; \; \mathrm{or} \; \; \tilde n := d_{\beta}^2 d_{\gamma}^2 (n - r^2 / 4m)$$ have valuation $0$. Here, $d_{\beta}$ and $d_{\gamma}$ denote the denominators of $\beta$ and $\gamma$, respectively.  Since $p \nmid \mathrm{det}(S)$, it follows that $d_{\beta}$ and $d_{\gamma}$ are invertible mod $p$; so we can multiply the congruence $$\lambda^2 m + q(v + \lambda \beta - \gamma) - r \lambda + n \equiv 0 \; (p^{\nu})$$ by $d_{\beta}^2 d_{\gamma}^2$ and replace $d_{\beta} d_{\gamma} v + \lambda d_{\beta} d_{\gamma} \beta - d_{\beta} d_{\gamma} \gamma$ by $v$ to obtain $$\mathbf{N}(p^{\nu}) = \# \Big\{(v,\lambda): \; d_{\beta}^2 d_{\gamma}^2 m \lambda^2 + q(v) - d_{\beta}^2 d_{\gamma}^2 r \lambda + d_{\beta}^2 d_{\gamma}^2 n \equiv 0 \, (p^{\nu}) \Big\}.$$ Here, $d_{\beta}^2 m, d_{\gamma}^2 n, d_{\beta} d_{\gamma} r \in \mathbb{Z}.$ By completing the square and replacing $\lambda - d_{\beta} \frac{d_{\gamma} d_{\beta} r}{2 d_{\beta}^2 m}$ by $\lambda$, we see that \begin{align*} \mathbf{N}(p^{\nu}) &= \# \Big\{ (v,\lambda) \in (\mathbb{Z}/p^{\nu}\mathbb{Z})^{e+1}: \; q(v) + d_{\beta}^2 m \lambda^2 + d_{\beta}^2 d_{\gamma}^2 (n - r^2 / 4m) \equiv 0 \, (p^{\nu}) \Big\} \\ &= \# \Big\{(v,\lambda)  \in (\mathbb{Z}/p^{\nu}\mathbb{Z})^{e+1}: \; v^T S v + 2 d_{\beta}^2 m \lambda^2 + 2 \tilde n \equiv 0 \, (p^{\nu}) \Big\}. \end{align*}

The polynomial $f(v,\lambda) = v^T S v + 2 d_{\beta}^2 m \lambda^2 + 2 \tilde n$ is $p$-integral and in isospectral normal form so proposition 20 (specifically, case 3) applies. The Igusa zeta function is $$\zeta_{Ig}(f;p;s) = \frac{1}{p^{e+1}} + I_{2\tilde n}(e+1, |\mathrm{det}(S)|)(s).$$ For even $e$, this is $$\zeta_{Ig}(f;p;s) = \Big[ 1 + p^{-e/2-1-s} \Big( \frac{\mathcal{D}'}{p} \Big) \Big] \cdot \frac{p - 1}{p - p^{-s}} - p^{-e/2 - 1} \Big( \frac{\mathcal{D}'}{p} \Big),$$ where $\mathcal{D}' = md_{\beta}^2 (-1)^{e/2 + 1} \tilde n \mathrm{det}(S)$, and after some algebraic manipulation we find that $$\frac{1 - p^{-s} \zeta_{Ig}(f;p;s)}{1 - p^{-s}} = \frac{1}{1 - p^{-s-1}} \Big[ 1 + \Big(\frac{\mathcal{D}'}{p} \Big) p^{-s-e/2-1} \Big]$$ and therefore $$L_p(k+e/2 - 1) = \frac{1}{1 - p^{-k + e/2 + 1}} \Big[ 1 + \Big( \frac{\mathcal{D}'}{p}\Big) p^{1-k} \Big) \Big].$$ For odd $e$, it is $$\zeta_{Ig}(f;p;s) = \frac{p-1}{p-p^{-s}} + p^{-(e+1)/2} \Big( \frac{D'}{p} \Big) \Big[ 1 - \frac{p - 1}{p - p^{-s}} \Big],$$ where $D' = 2 md_{\beta}^2 (-1)^{(e+1)/2} \mathrm{det}(S)$, and it follows that $$\frac{1 - p^{-s} \zeta_{Ig}(f;p;s)}{1 - p^{-s}} = \frac{1}{1 - p^{-s-1}} \Big[ 1 - \Big( \frac{D'}{p} \Big) p^{-s-(e+1)/2 - 1} \Big]$$ and therefore $$L_p(k + e/2 - 1) = \frac{1}{1 - p^{-k + e/2 + 1}} \Big[ 1 - \Big( \frac{D'}{p} \Big) p^{1/2 - k} \Big].$$
\end{rem}

\begin{prop} Define the constant $$\alpha_{k,m}(n,r) = \frac{(-1)^{(2k + b^+ - b^-)/4} \pi^{k-1/2} (4mn - r^2)^{k-3/2}}{2^{k-2} m^{k-1} \Gamma(k-1/2) \sqrt{|\mathrm{det}(S)|}}.$$ Define the set of ``bad primes'' to be $$\{2\} \cup \Big\{ p \, \mathrm{prime}: \; p | \mathrm{det}(S) \; \mathrm{or} \; p | d_{\beta}^2 m \; \mathrm{or} \; v_p(\tilde n) \ne 0\Big\}.$$ (i) If $e$ is even, then define  $$\mathcal{D} = \mathcal{D}' \cdot \prod_{\mathrm{bad} \, p} p^2 = md_{\beta}^2 (-1)^{e/2 + 1} \tilde n \mathrm{det}(S) \prod_{\mathrm{bad}\, p} p^2.$$ For $4mn - r^2 > 0$, $$c(n,r,\gamma) = \frac{\alpha_{k,m}(n,r) L_{\mathcal{D}}(k-1)}{\zeta(2k-2)} \prod_{\mathrm{bad}\, p} \Big[ \frac{1 - p^{-k + e/2 + 1}}{1 - p^{2-2k}} L_p(k + e/2 - 1) \Big].$$ (ii) If $e$ is odd, then define $$D = D' \cdot \prod_{\mathrm{bad} \, p} p^2 = 2md_{\beta}^2 (-1)^{(e+1)/2} \mathrm{det}(S) \prod_{\mathrm{bad}\, p} p^2.$$ For $4mn - r^2 > 0$, $$c(n,r,\gamma) = \frac{\alpha_{k,m}(n,r)}{L_D(k - 1/2)} \prod_{\mathrm{bad}\, p} \Big[ (1 - p^{-k + e/2 + 1}) L_p(k + e/2 - 1) \Big].$$
\end{prop}
 Here, $L_{\mathcal{D}}$ and $L_D$ denote the $L$-series $$L_{\mathcal{D}}(s) = \sum_{c=1}^{\infty} c^{-s} \Big( \frac{\mathcal{D}}{c} \Big), \; \; L_D(s) = \sum_{c=1}^{\infty} c^{-s} \Big( \frac{D}{c} \Big),$$ where $\Big( \frac{D}{c} \Big)$ and $\Big( \frac{\mathcal{D}}{c}\Big)$ is the Kronecker symbol. \\
\begin{proof} This follows immediately from the Euler products $$L_{\mathcal{D}}(s) = \prod_p \Big( 1 - \Big( \frac{\mathcal{D}}{p} \Big) p^{-s} \Big)^{-1}, \; \; L_D(s) = \prod_p \Big( 1 - \Big( \frac{D}{p} \Big) p^{-s} \Big)^{-1},$$ which are valid because $\mathcal{D}$ and $D$ are discriminants (congruent to $0$ or $1$ mod $4$) and therefore $\Big(\frac{\mathcal{D}}{a}\Big)$ and $\Big( \frac{D}{a}\Big)$ define Dirichlet characters of $a$ modulo $|\mathcal{D}|$ resp. $|D|$.
\end{proof}
In particular, $c(n,r,\gamma)$ is always rational. \\

The factors $L_p(k+e/2 - 1)$ are easy to evaluate for bad primes $p \ne 2$ using proposition 20. To calculate the factor at $p = 2$, we need a longer formula. This is described in the appendix.

\section{Poincar\'e square series of weight $5/2$}

An application of the Hecke trick shows that the Poincar\'e square series of weight $3$ is still the zero-value of the Jacobi Eisenstein series of weight $3$. This result is not surprising and the derivation is essentially the same as the weight $5/2$ case below, so we omit the details. However, the result in the case $k = 5/2$ is somewhat more complicated. \\

\begin{defn} For $k = 5/2$, we define the nonholomorphic Jacobi Eisenstein series of weight $5/2$, twisted at $\beta \in \Lambda'/\Lambda$, of index $m \in \mathbb{Z} - q(\beta),$ by $$E^*_{5/2,m,\beta}(\tau,z,s) = \frac{1}{2} \sum_{c,d} (c \tau + d)^{-5/2} |c \tau + d|^{-2s} \sum_{\lambda \in \mathbb{Z}} \mathbf{e} \Big( m \lambda^2 (M \cdot \tau) + \frac{2m \lambda z}{c \tau + d} - \frac{cmz^2}{c \tau + d} \Big) \rho^*(M)^{-1} \sigma_{\beta}^*(\lambda,0,0)^{-1} \mathfrak{e}_0.$$ This defines a holomorphic function of $s$ in the half-plane $\mathrm{Re}[s] > 0.$
\end{defn}

We write the Fourier series of $E_{5/2,m,\beta}^*$ in the form $$E_{5/2,m,\beta}^*(\tau,z,s) = \sum_{n,r,\gamma} c(n,r,\gamma,s,y) q^n \zeta^r \mathfrak{e}_{\gamma}.$$ As before, the contribution from $c = 0$ and $d = \pm 1$ is $$\sum_{\lambda \in \mathbb{Z}} \mathbf{e}\Big( m \lambda^2 \tau + 2 m \lambda z\Big) \mathfrak{e}_{\lambda \beta}.$$ (Here, the coefficients depend on $y$, since $E_{5/2,m,\beta}^*$ is not holomorphic in $\tau$.) We denote the contribution from all other terms by $c'(n,r,\gamma,s,y)$, so $$E_{5/2,m,\beta}^*(\tau,z,s) = \sum_{\lambda \in \mathbb{Z}} \mathbf{e}\Big( m \lambda^2 \tau + 2m \lambda z \Big) \mathfrak{e}_{\lambda \beta} + \sum_{n,r,\gamma} c'(n,r,\gamma,s,y) q^n \zeta^r \mathfrak{e}_{\gamma}.$$ A derivation similar to section 5 gives $$c'(n,r,\gamma,s,y) = \frac{1}{2 \sqrt{2im}} \sum_{c \ne 0} \frac{\sqrt{i}^{(b^- - b^+) \mathrm{sgn}(c)}}{|c|^{e/2} \sqrt{|\Lambda'/\Lambda|}} c^{-5/2} |c|^{-2s} K_c(\beta,m,\gamma,n,r) \int_{-\infty + iy}^{\infty +iy} \tau^{-2} |\tau|^{-2s} \mathbf{e}\Big( \tau (r^2 / 4m - n) \Big) \, \mathrm{d}x.$$

Substituting $\tau = y (t + i)$ in the integral yields \begin{align*} &\quad \int_{-\infty + iy}^{\infty + iy} \tau^{-2} |\tau|^{-2s} \mathbf{e} \Big( \tau (r^2 / 4m - n) \Big) \, \mathrm{d}x \\ &= y^{-1-2s} \mathbf{e}\Big( iy (r^2 / 4m - n) \Big) \int_{-\infty}^{\infty} (t + i)^{-2} (t^2 + 1)^{-s} \mathbf{e}\Big( yt (r^2 / 4m - n) \Big) \, \mathrm{d}t. \end{align*} We use $$\sqrt{i}^{(b^- - b^+) \mathrm{sgn}(c)} \mathrm{sgn}(c)^{-5/2} = (-1)^{(5 - b^- + b^+)/4} i^{5/2} = (-1)^{(1 - b^- + b^+)/4} \sqrt{i}$$ and conclude that \begin{align*} c'(n,r,\gamma,s,y) &= \frac{(-1)^{(1 + b^+ - b^-) / 4}}{\sqrt{2m |\Lambda'/\Lambda|}} I(y,r^2 / 4m - n, s) \sum_{c=1}^{\infty} c^{-5/2 - 2s - e/2} K_c(\beta,m,\gamma,n,r) \\ &= \frac{(-1)^{(1+ b^+ - b^-)/4}}{\sqrt{2m |\Lambda'/\Lambda|}} I(y,r^2 / 4m - n, s) \tilde L(5/2 + e/2 + 2s), \end{align*} where $I(y,N,s)$ denotes the integral $$I(y,N,s) = y^{-1-2s} e^{-2\pi Ny} \int_{-\infty}^{\infty} (t + i)^{-2} (t^2 + 1)^{-s} \mathbf{e}( Nyt) \, \mathrm{d}t,$$ and $$\tilde L(s) = \sum_{c=1}^{\infty} c^{-s} K_c(\beta,m,\gamma,n,r)$$ as before.

\begin{rem} When $r^2 \ne 4mn$, we were able to express $\tilde L(s)$ up to finitely many holomorphic factors as $\frac{1}{L_D(s - e/2 - 1/2)}$, and it follows that $\tilde L(s)$ is holomorphic in $5/2 + e/2$. In particular, if $r^2 \ne 4mn$, then the coefficient $c'(n,r,\gamma,0,y)$ is independent of $y$ and given by $$c'(n,r,\gamma,0,y) = \frac{\alpha_{k,m}(n,r)}{L_D(2)} \prod_{\mathrm{bad}\, p} \Big[ \frac{1 - p^{-3/2 + e/2}}{1 - \Big( \frac{D}{p} \Big) p^{-2}} L_p(3/2 + e/2) \Big] \; \mathrm{if} \; 4mn-r^2 > 0,$$ and $c'(n,r,\gamma,0,y) = 0$ if $4mn - r^2 < 0,$ just as for $k \ge 3.$ This analysis does not apply when $r^2 = 4mn$ and indeed $\tilde L$ may have a (simple) pole in $5/2 + e/2$ in that case.
\end{rem}

We will study the coefficients $c'(n,r,\gamma,0,y)$ when $4mn = r^2$. The integral $I(y,0,s)$ is zero at $s = 0$, and its derivative there is $$\frac{\partial}{\partial s} \Big|_{s=0} I(y,0,s) = -y^{-1} \int_{-\infty}^{\infty} (t+i)^{-2} \log(t^2 + 1) \, \mathrm{d}t = -\frac{\pi}{y}.$$ This cancels the possible pole of $\tilde L(5/2 + e/2 + 2s)$ at $0$, and therefore we need to know the residue of $\tilde L(5/2 + e/2 + 2s)$ there. As before, $\tilde L$ factors as $$\tilde L(5/2 + e/2 + 2s) = \zeta(2s + 3/2 - e/2)^{-1} L(3/2 + e/2 + 2s)$$ where $L(s)$ has an Euler product $$L(s) = \prod_{p \, \mathrm{prime}} L_p(s), \; \mathrm{with} \; L_p(s) = \sum_{\nu=0}^{\infty} \mathbf{N}(p^{\nu}) p^{-\nu s},$$ and $\mathbf{N}(p^{\nu})$ is the number of zeros of the polynomial $f(v,\lambda) = q(v + \lambda \beta - \gamma) + m \lambda^2 - r \lambda + n$ mod $p^{\nu}.$

\begin{rem} Identify $\Lambda = \mathbb{Z}^n$ and $q(v) = \frac{1}{2} v^T S v$ where $S$ is the Gram matrix. We will calculate $L_p(s)$ for primes $p$ dividing neither $\mathrm{det}(S)$ nor $d_{\beta}^2 m$. In this case, it follows that $$\mathbf{N}(p^{\nu}) = \#\Big\{ (v,\lambda) \in (\mathbb{Z}/p^{\nu} \mathbb{Z})^{e+1}: \; v^T S v + 2 d_{\beta}^2 m \lambda^2 \equiv 0 \, (p^{\nu}) \Big\}.$$ We are in case (i) of proposition 20 and it follows that $$\zeta_{Ig}(f;p;s) = \Big[1 - p^{-(e+1)/2} \Big( \frac{D'}{p} \Big) \Big] \cdot \Big[ 1 + p^{-s - (e+1)/2} \Big( \frac{D'}{p} \Big) \Big] \cdot \frac{p-1}{(p-p^{-s})(1 - p^{-2s - e - 1})}$$ with $D' = m d_{\beta}^2 (-1)^{(e+1)/2} \mathrm{det}(S).$ After some algebraic manipulation, we find that $$\frac{1 - p^{-s} \zeta_{Ig}(f;p;s)}{1 - p^{-s}} = \frac{1 - \Big( \frac{D'}{p} \Big) p^{-s-1-(e+1)/2}}{(1 - p^{-s-1}) (1 - \Big( \frac{D'}{p} \Big) p^{-s-(e+1)/2})},$$ so $$L_p(3/2 + e/2 + 2s) = \frac{1 - \Big( \frac{D'}{p} \Big) p^{-2-2s}}{(1 - p^{e/2 - 3/2 - 2s}) (1 - \Big( \frac{D'}{p} \Big) p^{-1 - 2s})}.$$ 
\end{rem}
This immediately implies the following lemma: 
\begin{lem} In the situation treated in this section, define $D = D' \cdot \prod_{\mathrm{bad}\, p} p^2$; then $$\tilde L(5/2 + e/2 + 2s) = \frac{L_D(2s+1)}{L_D(2s + 2)} \prod_{\mathrm{bad}\, p} \Big[ (1 - p^{e/2 - 3/2 - 2s}) L_p(3/2 + e/2 + 2s) \Big].$$
\end{lem}

Notice that $L_D(2s+1)$ is holomorphic in $s = 0$ unless $D$ is a square, in which case it is the Riemann zeta function with finitely many Euler factors missing.

\begin{prop} If $4mn - r^2 = 0$, then $c'(n,r,\gamma,0,y) = 0$ unless $D$ is a square, in which case $$c'(n,r,\gamma,0,y) = \frac{(-1)^{(1 + b^+ - b^-)/4}}{\sqrt{2m |\Lambda' / \Lambda}|} \cdot \frac{3}{\pi y} \prod_{p | D} \Big[ (1 - p^{(e-3)/2}) L_p((e+3)/2) \Big].$$
\end{prop}
\begin{proof} Assume that $D$ is a square. As $s \rightarrow 0$, $$\lim_{s \rightarrow 0} c'(n,r,\gamma,s,y) = \frac{(-1)^{(1+b^+ - b^-)/4}}{\sqrt{2m |\Lambda'/\Lambda|}} \cdot \frac{\partial}{\partial s}\Big|_{s=0} I(y,0,s) \cdot \mathrm{Res}\Big( \tilde L(5/2 + e/2 + 2s);s=0\Big).$$ We calculated $$\frac{\partial}{\partial s}\Big|_{s=0} I(y,0,s) = -\frac{\pi}{y}$$ earlier. The residue of $\tilde L(5/2+e/2+2s)$ at $0$ is $$\frac{1}{L_D(2)} \prod_{\mathrm{bad}\, p} \Big[ (1 - p^{e/2 - 3/2 }) L_p(3/2 + e/2)\Big] \cdot \mathrm{Res}(L_D(2s+1);s=0),$$ and using $$L_D(2s+1) = \zeta(2s+1) \prod_{p | D} (1 - p^{-2s-1})$$ and the fact that $\zeta(s)$ has residue $1$ at $s=1$, it follows that $$\mathrm{Res}(L_D(2s+1);s=0) = \frac{1}{2} \prod_{p | D} (1 - p^{-1}).$$ We write $$L_D(2) = \zeta(2) \prod_{p | D} (1 - p^{-2}) = \frac{\pi^2}{6} \prod_{p | D} (1 - p^{-2}).$$ Since the ``bad primes'' are exactly the primes dividing $D$ (by construction of $D$), we find $$\mathrm{Res}\Big( \tilde L(5/2 + e/2 + 2s);s=0\Big) = \frac{3}{\pi^2} \prod_{p | D} \Big[ \frac{(1 - p^{e/2 - 3/2}) (1 - p^{-1})}{1 - p^{-2}} L_p(3/2 + e/2) \Big],$$ which gives the formula.
\end{proof}

Denote the constant in proposition 28 by $$A_n = \frac{(-1)^{(5 + b^+ - b^-)/4}}{\sqrt{2m |\Lambda' / \Lambda}|} \cdot \frac{3}{\pi} \prod_{p | D} \Big[ \frac{(1 - p^{(e-3)/2})(1 - p^{-1})}{1 - p^{-2}} L_p((e+3)/2) \Big],$$ such that $E_{5/2,m,\beta}^*(\tau,z)- \frac{1}{y} \vartheta$ is holomorphic, where $\vartheta$ is the theta function $$\vartheta(\tau,z) = \sum_{\gamma \in \Lambda'/\Lambda} \sum_{\substack{4mn - r^2 = 0 \\ n \in \mathbb{Z} - q(\gamma) \\ r \in \mathbb{Z} - \langle \gamma, \beta \rangle}} A_n q^n \zeta^r \mathfrak{e}_{\gamma}.$$ Even when $D$ is not square, this becomes true after defining $A_n = 0$ for all $n$.

\begin{lem} $$\vartheta(\tau,z) = \sum_{\gamma \in \Lambda'/\Lambda} \sum_{\substack{4mn - r^2 = 0 \\ n \in \mathbb{Z} - q(\gamma) \\ r \in \mathbb{Z} - \langle \gamma, \beta \rangle}} A_n q^n \zeta^r \mathfrak{e}_{\gamma}$$ is a Jacobi form of weight $1/2$ and index $m$ for the representation $\rho_{\beta}^*$.
\end{lem}
\begin{proof} We give a proof relying on the transformation law of $E_{5/2,m,\beta}^*.$ Denote by $$E_{5/2,m,\beta}(\tau,z) = E_{5/2,m,\beta}^*(\tau,z,0) - \frac{1}{y} \vartheta(\tau,z)$$ the holomorphic part of $E_{5/2,m,\beta}^*.$ For any $M = \begin{pmatrix} a & b \\ c & d \end{pmatrix} \in \tilde \Gamma$, \begin{align*} &\quad E_{5/2,m,\beta}\Big( \frac{a \tau + b}{c \tau + d}, \frac{z}{c \tau + d} \Big) + \frac{|c\tau+d|^2}{y} \vartheta\Big( \frac{a \tau + b}{c \tau + d}, \frac{z}{c \tau + d} \Big) \\ &= E_{5/2,m,\beta}^*\Big( \frac{a \tau + b}{c \tau + d}, \frac{z}{c \tau + d}, 0 \Big) \\ &= (c \tau + d)^{5/2} \mathbf{e}\Big( \frac{mcz^2}{c \tau + d} \Big) \rho^*(M) E_{5/2,m,\beta}^*(\tau,z,0) \\ &= (c \tau + d)^{5/2} \mathbf{e}\Big( \frac{mcz^2}{c \tau + d} \Big) \rho^*(M) E_{5/2,m,\beta}(\tau,z) + \frac{(c \tau + d)^{5/2}}{y} \mathbf{e}\Big( \frac{mcz^2}{c \tau + d} \Big) \rho^*(M) \vartheta(\tau,z). \end{align*} In particular, $$\frac{1}{y}\Big[ |c\tau + d|^2  - (c \tau + d)^{5/2} \mathbf{e}\Big(\frac{mcz^2}{c \tau + d} \Big) \rho^*(M) \Big] \vartheta(\tau,z)$$ is holomorphic. Differentiating twice with respect to $\overline{\tau}$, we get $$\frac{1}{y} \Big[ (c \tau + d)^2  - (c \tau + d)^{5/2} \mathbf{e}\Big( \frac{mcz^2}{c \tau + d} \Big) \rho^*(M) \Big] \vartheta(\tau,z) = 0,$$ which implies the modularity of $\vartheta$ under $M$. It is easy to verify the transformation law under the Heisenberg group by a similar argument.
\end{proof}

We can now compute $Q_{5/2,m,\beta}$. Let $\vartheta(\tau)$ denote the zero-value $\vartheta(\tau,0)$.

\begin{prop} The Poincar\'e square series of weight $5/2$ is $$Q_{5/2,m,\beta}(\tau) = E_{5/2,m,\beta}(\tau,0) + 4i \vartheta'(\tau).$$
\end{prop}
\begin{proof} Using the modularity of $E_{5/2,m,\beta}^*$ and $\vartheta$, we find that $E_{5/2,m,\beta}(\tau,0)$ transforms under $\tilde \Gamma$ by $$E_{5/2,m,\beta}\Big( \frac{a \tau + b}{c \tau + d}, 0 \Big) = \rho^*(M) \Big[ (c \tau + d)^{5/2} E_{5/2,m,\beta}(\tau,0) - 2ic (c \tau + d)^{3/2} \vartheta(\tau) \Big].$$ Differentiating the equation $\vartheta(M \cdot \tau) = (c\tau + d)^{1/2} \rho^*(M) \vartheta(\tau)$ gives the similar equation $$\vartheta'(M \cdot \tau) = \rho^*(M) \Big[ (c \tau + d)^{5/2} \vartheta'(\tau) + \frac{1}{2} c (c \tau + d)^{3/2} \vartheta(\tau) \Big].$$ This implies that $E_{5/2,m,\beta}(\tau,0) + 4i \vartheta'(\tau)$ is a modular form of weight $5/2$. \\

Now we prove that it equals $Q_{5/2,m,\beta}$ by showing that it satisfies the characterization of $Q_{5/2,m,\beta}$ with respect to the Petersson scalar product. First, we remark that $E_{5/2,m,\beta}^*(\tau,0,0)$, although not holomorphic, satisfies that characterization: for any cusp form $f(\tau) = \sum_{\gamma} \sum_n c(n,\gamma) q^n$, and any $\mathrm{Re}[s] > 0$, $$\Big\langle f(\tau), E_{5/2,m,\beta}^*(\tau,0,s) \Big\rangle y^{1/2 + 2s} \, \mathrm{d}x \, \mathrm{d}y$$ is invariant under $\tilde \Gamma,$ and we integrate: \begin{align*} &\quad \int_{\tilde \Gamma \backslash \mathbb{H}} \langle f(\tau), E_{5/2,m,\beta}^*(\tau,0,s) \rangle y^{1/2 + 2s} \, \mathrm{d}x \, \mathrm{d}y \\ &= \sum_{\gamma \in \Lambda'/\Lambda} \sum_{\lambda \in \mathbb{Z}} \sum_n \int_{-1/2}^{1/2} \int_0^{\infty} \langle c(n,\gamma) \mathfrak{e}_{\gamma}, \mathfrak{e}_{\lambda \beta} \rangle \mathbf{e}\Big( n(x+iy) - m \lambda^2 (x - iy) \Big) y^{1/2 + 2s} \, \mathrm{d}x \, \mathrm{d}y \\ &= \sum_{\lambda \ne 0} c(\lambda^2 m, \lambda \beta) \int_0^{\infty} e^{-4\pi m \lambda^2 y} y^{1/2 + 2s} \, \mathrm{d}y \\ &= \sum_{\lambda \ne 0} c(\lambda^2 m, \lambda \beta) \frac{\Gamma(3/2 + 2s)}{(4\pi m \lambda^2)^{3/2 + 2s}}. \end{align*} Taking the limit as $s \rightarrow 0$, we get $$\lim_{s \rightarrow 0} \int_{\tilde \Gamma \backslash \mathbb{H}} \langle f(\tau), E_{5/2,m,\beta}^*(\tau,0,s) \rangle y^{1/2 + 2s} \, \mathrm{d}x \, \mathrm{d}y = \sum_{\lambda \ne 0} c(\lambda^2 m, \lambda \beta) \frac{\Gamma(3/2)}{(4\pi m \lambda^2)^{3/2}}.$$

The difference $$ \Big( E_{5/2,m,\beta}(\tau,0) + 4i \vartheta'(\tau) \Big) - E_{5/2,m,\beta}^*(\tau,0,0) = 4i \vartheta'(\tau) + \frac{1}{y} \vartheta(\tau)$$ is orthogonal to all cusp forms, because: when we integrate against a Poincar\'e series $$P_{5/2,n,\gamma}(\tau) = \frac{1}{2}\sum_{c,d} (c \tau + d)^{-k} \mathbf{e}\Big( n (M \cdot \tau) \Big) \rho^*(M)^{-1}(\mathfrak{e}_{\gamma}),$$ we find that \begin{align*} &\quad \Big( 4i \vartheta' + \frac{1}{y} \vartheta, P_{5/2,n,\gamma} \Big) \\ &= 4i \int_{-1/2}^{1/2} \int_0^{\infty} \langle \vartheta'(\tau), \mathbf{e}(n \tau) \mathfrak{e}_{\gamma} \rangle y^{1/2} \, \mathrm{d}y \, \mathrm{d}x + \int_{-1/2}^{1/2} \int_0^{\infty} \langle \vartheta(\tau), \mathbf{e}(n \tau) \mathfrak{e}_{\gamma} \rangle y^{-1/2} \, \mathrm{d}y \, \mathrm{d}x \\ &= \sum_{r \in \mathbb{Z} - \langle \beta, \gamma \rangle} \delta_{4mn - r^2} A_n \Big( 4i \cdot (2\pi i n) \frac{\Gamma(3/2)}{(4\pi n)^{3/2}} + \frac{\Gamma(1/2)}{(4\pi n)^{1/2}}\Big) \\ &= 0, \end{align*} since $4i \cdot (2\pi i n) \frac{\Gamma(3/2)}{(4\pi n)^{3/2}} + \frac{\Gamma(1/2)}{(4\pi n)^{1/2}} = 0$ for all $n$. Here, $\delta_N$ denotes the delta function $\delta_N = \begin{cases}1 : & N = 0; \\ 0: & N \ne 0. \end{cases}$ \\

Finally, the fact that $E_{5/2,m,\beta}(\tau,0) + 4i \vartheta'(\tau)$ and $Q_{5/2,m,\beta}$ both have constant term $1 \cdot \mathfrak{e}_0$ implies that their difference is a cusp form that is orthogonal to all Poincar\'e series and therefore zero.
\end{proof}

\begin{ex} Consider the quadratic form with Gram matrix $S = \begin{pmatrix} -2 \end{pmatrix}.$ The space of weight $5/2$ modular forms is $1$-dimensional, spanned by the Eisenstein series $$E_{5/2}(\tau) = \Big( 1 - 70q - 120q^2 - ... \Big) \mathfrak{e}_0 + \Big( -10 q^{1/4} - 48 q^{5/4} - 250q^{9/4} - ... \Big) \mathfrak{e}_{1/2}.$$ The nonmodular Jacobi Eisenstein series of index $1$ and weight $5/2$ is \begin{align*} E_{5/2,1,0}(\tau,z) &= \Big( 1 + q(\zeta^{-2} - 16 \zeta^{-1} - 16 - 16\zeta + \zeta^2) + q^2 (\zeta^{-2} - 32 \zeta^{-1} - 24 - 32 \zeta + \zeta^2) + ... \Big) \mathfrak{e}_0 \\ &+ \Big( -4q^{1/4} + q^{5/4} (-4 \zeta^{-2} - 8 \zeta^{-1} - 24 - 8 \zeta - 4\zeta^2) + ... \Big) \mathfrak{e}_{1/2}, \end{align*} and setting $z = 0$, we find $$E_{5/2,1,0}(\tau,0) = \Big( 1 - 46q - 120q^2 - 240q^3 - 454q^4 - ... \Big) \mathfrak{e}_0 + \Big( -4q^{1/4} - 48q^{5/4} - 196q^{9/4} - 240q^{13/4} - ... \Big) \mathfrak{e}_{1/2}.$$ This differs from $E_{5/2}$ by exactly $$\Big( -24q - 96q^4 - ... \Big) \mathfrak{e}_0 + \Big( -6 q^{1/4} - 54 q^{9/4} - ... \Big) \mathfrak{e}_{1/2} = 4i \vartheta'(\tau).$$ For comparison, the Jacobi Eisenstein series of index $2$ (which is a true Jacobi form) is \begin{align*} E_{5/2,2,0}(\tau,z) &= \Big( 1 + q (-10 \zeta^{-2} -16 \zeta^{-1} - 18 - 16 \zeta - 10 \zeta^2) + \\ &\quad\quad + q^2 (\zeta^{-4} - 16 \zeta^{-3} - 12 \zeta^{-2} - 16 \zeta^{-1} -34 - 16 \zeta - 12 \zeta^2 - 16 \zeta^3 + \zeta^4 ) + \cdots \Big) \mathfrak{e}_0 \\ &+  \Big( q^{1/4} (-2 \zeta^{-1} - 6 - 2\zeta) + q^{5/4} (-2\zeta^{-3} - 4\zeta^{-2} -14\zeta^{-1} -8 - 14\zeta - 4\zeta^2 - 2\zeta^3) + \cdots \Big) \mathfrak{e}_{1/2}, \end{align*} and we see that $E_{5/2,2}(\tau,0) = Q_{5/2,2}(\tau) = E_{5/2}(\tau)$ as predicted.
\end{ex}

\section{Coefficient formula for $Q_{k,m,\beta}$}

For convenience, the results of the previous sections are summarized here.

\begin{prop} Let $k \ge 5/2$. The coefficients $c(n,\gamma)$ of the Poincar\'e square series $Q_{k,m,\beta}$, $$Q_{k,m,\beta}(\tau) = \sum_{\gamma \in \Lambda'/\Lambda} \sum_{n \in \mathbb{Z} - q(\gamma)} c(n,\gamma) q^n \mathfrak{e}_{\gamma},$$ are given as follows:\\ (i) If $n < 0$, then $c(n,\gamma) = 0$. \\ (ii) If $n = 0$, then $c(n,\gamma) = 1$ if $\gamma = 0$ and $c(n,\gamma) = 0$ otherwise. \\ (iii) If $n > 0$, then \begin{align*} c(n,\gamma) &=  \delta + \frac{(-1)^{(2k - b^- + b^+) / 4} \pi^{k-1/2}}{2^{k-2} m^{k-1} \Gamma(k-1/2) \zeta(2k-2) \sqrt{|\mathrm{det}(S)|}} \times \\ &\quad\quad\quad \times \sum_{-\sqrt{4mn} < r < \sqrt{4mn}} \Big( L_{\mathcal{D}}(k-1) \prod_{\mathrm{bad}\, p} \Big[ \frac{1 - p^{-k + e/2 + 1}}{1 - p^{2-2k}} L_p(k + e/2 - 1) \Big] \Big)\end{align*} if $e$ is even, and \begin{align*} c(n,\gamma) &=  \varepsilon_{5/2} + \delta + \frac{(-1)^{(2k - b^- + b^+) / 4} \pi^{k-1/2}}{2^{k-2} m^{k-1} \Gamma(k-1/2) \sqrt{|\mathrm{det}(S)|}} \times \\ &\quad\quad\quad \times \sum_{-\sqrt{4mn} < r < \sqrt{4mn}} \Big( \frac{1}{L_D(k-1/2)} \prod_{\mathrm{bad}\, p} \Big[ (1 - p^{-k + e/2 + 1}) L_p(k + e/2 - 1) \Big] \Big)\end{align*} if $e$ is odd. Here, for each $r$, we define the set of ``bad primes'' to be $$\{ \mathrm{bad} \, \mathrm{primes}\} = \{2\} \cup \Big\{p \, \mathrm{prime}: \; p | d_{\beta}^2 m \mathrm{det}(S) \; \mathrm{or} \; v_p(d_{\beta}^2 d_{\gamma}^2 (n - r^2 / 4m)) \ne 0 \Big\},$$ and we define $$\mathcal{D} = 2md_{\beta}^4 d_{\gamma}^2 (-1)^{e/2 + 1} (n - r^2 / 4m) \mathrm{det}(S) \prod_{\mathrm{bad}\, p} p^2$$ if $e$ is even and $$D = m d_{\beta}^2 (-1)^{(e+1)/2} \mathrm{det}(S) \prod_{\mathrm{bad}\, p} p^2$$ if $e$ is odd; $L_{\mathcal{D}}$ and $L_D$ denote the $L$-series $$L_{\mathcal{D}}(s) = \sum_{a=1}^{\infty} \Big( \frac{\mathcal{D}}{a} \Big) a^{-s}, \; \; L_D(s) = \sum_{a=1}^{\infty} \Big( \frac{D}{a} \Big) a^{-s};$$ and $L_p$ is the $L$-series $$L_p(s) = \sum_{\nu=0}^{\infty} \mathbf{N}(p^{\nu}) p^{-\nu s},$$ where $$\mathbf{N}(p^{\nu}) = \#\Big\{ (v,\lambda) \in \mathbb{Z}^{n+1} / p^{\nu} \mathbb{Z}^{n+1}: \; q(v + \lambda \beta - \gamma) + m \lambda^2 - r \lambda + n = 0 \in \mathbb{Z}/p^{\nu} \mathbb{Z} \Big\}.$$ Finally, $$\delta = \begin{cases} 2: & n = m\lambda^2 \; \mathrm{for} \; \mathrm{some} \; \lambda \in \mathbb{Z}, \; \mathrm{and} \; \gamma = \lambda \beta; \\ 0: & \mathrm{otherwise}; \end{cases}$$ and $\varepsilon_{5/2} = 0$ unless $k = 5/2,$ in which case $$\varepsilon_{5/2} = \begin{cases} \sum_{\substack{r \in \mathbb{Z} - \langle \gamma, \beta \rangle \\ r^2 = 4mn}} \frac{24 n \cdot (-1)^{(5 + b^+ - b^-) / 4}}{\sqrt{2m \cdot \mathrm{det}(S)}} \prod_{\mathrm{bad}\, p} \Big[ \frac{(1 - p^{(e-3)/2}) (1 - p^{-1})}{1 - p^{-2}} L_p((e+3)/2) \Big]: & D = \square \ \\ 0: & \mathrm{otherwise}; \end{cases}$$ where $D = \square$ means that $D$ is a rational square.
\end{prop}

\begin{proof} For $k > 5/2$, since $Q_{k,m,\beta}(\tau) = E_{k,m,\beta}(\tau,0)$, we get the coefficients of $Q_{k,m,\beta}$ by summing the coefficients of $E_{k,m,\beta}$ over $r$. $\delta$ accounts for the contribution from the term $$\sum_{\lambda \in \mathbb{Z}} \mathbf{e}\Big( m \lambda^2 \tau + 2m \lambda z\Big) \mathfrak{e}_{\lambda \beta}.$$ When $k = 5/2$, $\varepsilon_{5/2}$ accounts for $4i$ times the derivative of the theta series \[ \vartheta(\tau) = \sum_{\gamma \in \Lambda'/\Lambda} \sum_{\substack{4mn - r^2 = 0 \\ n \in \mathbb{Z} - q(\gamma) \\ r \in \mathbb{Z} - \langle \gamma,\beta \rangle}} A_n q^n \mathfrak{e}_{\gamma}. \qedhere \]
\end{proof}

\section{Example - calculating an automorphic product}

The notation in this section is taken from \cite{Bn}. \\

Since $Q_{k,m,\beta}$ can be calculated efficiently, we can automate the process of searching for automorphic products. The level of the lattice is irrelevant for this method, which seems to be an advantage over other ways of constructing automorphic products. \\

Let $\Lambda$ be an even lattice of signature $(2,n)$. Recall that Borcherds' singular theta correspondence (\cite{Bn}) sends a nearly-holomorphic modular form with integer coefficients $$f(\tau) = \sum_{\gamma} \sum_n c(n,\gamma) \mathfrak{e}_{\gamma}$$ of weight $k = 1 - n/2$ for the Weil representation to a meromorphic automorphic form $\Psi$ on the Grassmannian of $\Lambda$. The weight of $\Psi$ is $\frac{c(0,0)}{2}$, and $\Psi$ is holomorphic when $c(n,\gamma)$ is nonnegative for all $\gamma$ and $n < 0.$ \\

Automorphic products $\Psi$ of \textbf{singular weight} $n/2 - 1$ are particularly interesting, since in this case most of the Fourier coefficients of $\Psi$ must vanish: the nonzero Fourier coefficients correspond to vectors of norm zero. 

Tensoring nearly-holomorphic modular forms of weight $k$ for $\rho$ and weight $2-k$ for $\rho^*$ gives a scalar-valued (nearly-holomorphic) modular form of weight $2$, or equivalently an invariant differential form on $\mathbb{H}$, whose residue in $\infty$ must be $0$. This implies that the constant term in the Fourier expansion must be zero. Also, the coefficients $c(n,\gamma)$ of a nearly-holomorphic modular form must satisfy $c(n,\gamma) = c(n,-\gamma)$ for all $n$ and $\gamma$, due to the transformation law under $Z$. As shown in \cite{Bo} and \cite{Br}, this is the only obstruction for a sum $\sum_{n < 0} \sum_{\gamma} c(n,\gamma) \mathfrak{e}_{\gamma} + c(0,0) \mathfrak{e}_0$ to occur as the principal part of a nearly-holomorphic modular form. \\

The lattice $A_1(-2) + A_1(-2) + II_{1,1} + II_{1,1}$ produces an automorphic product of singular weight. This product also arises through an Atkin-Lehner involution from an automorphic product attached to the lattice $A_1 \oplus A_1 \oplus II_{1,1} \oplus II_{1,1}(8)$, found by Scheithauer in \cite{Sch}. \\

Using the dimension formula (proposition 6), for the lattice $\Lambda = \mathbb{Z}^2$ with Gram matrix $\begin{pmatrix} -4 & 0 \\ 0 & -4 \end{pmatrix}$, we find $$\mathrm{dim}\, M_3(\rho^*) = 4, \; \; \mathrm{dim}\, S_3(\rho^*) = 2.$$ The Eisenstein series of weight $3$ is \begin{align*} E_{3,(0,0)}(\tau) &= \Big( 1 - 24q - 164q^2 -192 q^3 - ... \Big) \mathfrak{e}_{(0,0)} \\ &+ \Big( -1/2 q^{1/8} - 73/2 q^{9/8} - 145 q^{17/8} - ... \Big) (\mathfrak{e}_{(1/4,0)} + \mathfrak{e}_{(3/4,0)} + \mathfrak{e}_{(0,1/4)} + \mathfrak{e}_{(0,3/4)} ) \\ &+ \Big( -10q^{1/2} - 48q^{3/2} - 260 q^{5/2} - ... \Big) (\mathfrak{e}_{(1/2,0)} + \mathfrak{e}_{(0,1/2)}) \\ &+ \Big( -2q^{1/4} - 52 q^{5/4} - 146q^{9/4} - ... \Big) (\mathfrak{e}_{(1/4,3/4)} + \mathfrak{e}_{(3/4,1/4)}) \\ &+ \Big( -13 q^{5/8} - 85 q^{13/8} - 192 q^{21/8} - ... \Big) (\mathfrak{e}_{(1/2,1/4)} + \mathfrak{e}_{(1/2,3/4)} + \mathfrak{e}_{(1/4,1/2)} + \mathfrak{e}_{(3/4,1/2)} ) \\ &+ \Big( -44q - 96q^2 - 288 q^3 - ... \Big) \mathfrak{e}_{(1/2,1/2)}. \end{align*} We find two linearly independent cusp forms as differences between $E_3$ and particular Poincar\'e square series: for example, \begin{align*} \frac{2}{3} \Big( Q_{3,1/8,(1/4,0)} - E_3 \Big) &= \Big( q^{1/8} + 9 q^{9/8} - 30 q^{17/8} + ... \Big) (\mathfrak{e}_{(1/4,0)} + \mathfrak{e}_{(3/4,0)} - \mathfrak{e}_{(0,1/4)} - \mathfrak{e}_{(0,3/4)}) \\ &+ \Big( 6q^{5/8} - 10q^{13/8} - 42 q^{29/8} - ... \Big) (\mathfrak{e}_{(1/2,1/4)} + \mathfrak{e}_{(1/2,3/4)} - \mathfrak{e}_{(1/4,1/2)} - \mathfrak{e}_{(3/4,1/2)} ) \\ &+ \Big( 8 q^{1/2} - 48 q^{5/2} + 72 q^{9/2} + ... \Big) (\mathfrak{e}_{(1/2,0)} - \mathfrak{e}_{(0,1/2)}), \end{align*} and $$\frac{1}{3}\Big( Q_{3,1/4,(1/4,1/4)} - E_3 \Big) = \Big( q^{1/4} - 6 q^{5/4} + 9 q^{9/4} + 10q^{13/4} + ... \Big) (\mathfrak{e}_{(1/4,1/4)} + \mathfrak{e}_{(3/4,3/4)} - \mathfrak{e}_{(1/4,3/4)} - \mathfrak{e}_{(3/4,1/4)}).$$ The other Eisenstein series $E_{3,(1/2,1/2)}$ can be easily computed by averaging $E_{3,(0,0)}$ over the Schr\"odinger representation (as in the appendix), but Eisenstein series other than $E_{k,0}$ never represent new obstructions so we do not need them. \\

We see that the sum $$q^{-1/8} ( \mathfrak{e}_{(1/4,0)} + \mathfrak{e}_{(3/4,0)} + \mathfrak{e}_{(0,1/4)} + \mathfrak{e}_{(0,3/4)}) + 2 \mathfrak{e}_{(0,0)}$$ occurs as the principal part of a nearly-holomorphic modular form, and the corresponding automorphic product has weight $1$ (which is the singular weight for the lattice $\Lambda \oplus II_{1,1} \oplus II_{1,1}$ of signature $(2,4)$). \\

A brute-force way to calculate the nearly-holomorphic modular form $F$ is to search for $\Delta \cdot F$ among cusp forms of weight $11$ for $\rho$. Since $\rho$ is also the dual Weil representation $\rho^*$ of the lattice with Gram matrix $\begin{pmatrix} 4 & 0 \\ 0 & 4 \end{pmatrix}$, we can use the same formulas for Poincar\'e square series. This is somewhat messier since the cusp space is now $8$-dimensional. Using the coefficients $$\alpha_0 = \frac{1222146606526920765211168}{665492278281307137675}, \; \alpha_1 = -\frac{814700552816424434236}{1996476834843921413025}, \; \alpha_2 = -\frac{5383641094234426568192}{133098455656261427535},$$ $$\alpha_3 = \frac{77190276919058739618292}{665492278281307137675}, \; \alpha_4 = -\frac{3816441333371605691531264}{1996476834843921413025},$$ a calculation shows that \begin{align*} &\quad F = \frac{\alpha_0 E_{11,0} + \alpha_1 Q_{11,1,0} + \alpha_2 Q_{11,2,0} + \alpha_3 Q_{11,3,0} + \alpha_4 Q_{11,4,0}}{\Delta} \\ &= \Big( 2 + 8q + 24 q^2 + 64 q^3 + 152 q^4 + ... \Big) (\mathfrak{e}_{(0,0)} - \mathfrak{e}_{(1/2,1/2)}) \\ &+ \Big( q^{-1/8} + 3 q^{7/8} + 11 q^{15/8} + 28 q^{23/8} + ... \Big) (\mathfrak{e}_{(1/4,0)} + \mathfrak{e}_{(3/4,0)} + \mathfrak{e}_{(0,1/4)} + \mathfrak{e}_{(0,3/4)}) \\ &+ \Big( -2q^{3/8} - 6q^{11/8} - 18 q^{19/8} - 44q^{27/8} - ... \Big) (\mathfrak{e}_{(1/4,1/2)} + \mathfrak{e}_{(3/4,1/2)} + \mathfrak{e}_{(1/2,1/4)} + \mathfrak{e}_{(1/2,3/4)}). \end{align*} Once enough coefficients have been calculated, it is not hard to identify these components: the coefficients come from the weight $-1$ eta products $$\frac{2 \eta(2\tau)^2}{\eta(\tau)^4} = 2 + 8q + 24 q^2 + 64 q^3 + 152 q^4 + ...$$ and $$\frac{\eta(\tau/2)^2}{\eta(\tau)^4} = q^{-1/8} - 2q^{3/8} + 3 q^{7/8} - 6 q^{11/8} + 11 q^{15/8} - 18 q^{19/8} + ...$$

We will calculate the automorphic product using theorem 13.3 of \cite{Bn}, following the pattern of the examples of \cite{DHM}. Fix the primitive isotropic vector $z = (1,0,0,0,0,0)$ and $z' = (0,0,0,0,0,1)$ and the lattice $K = \Lambda \oplus II_{1,1}$. We fix as positive cone the component of positive-norm vectors containing those of the form $(+,\ast,\ast,+)$. This is split into Weyl chambers by the hyperplanes $\alpha^{\perp}$ with $\alpha \in \{\pm (0,1/4,0,0), \pm (0,0,1/4,0)\}$. These are all essentially the same so we will fix the Weyl chamber $$W = \{(x_1,x_2,x_3,x_4): \; x_1,\, x_2,\, x_3, \, x_4,\, x_1x_4 - 2x_2^2 - 2x_3^2 > 0\} \subseteq K \otimes \mathbb{R}.$$ The Weyl vector attached to $F$ and $W$ is the isotropic vector $$\rho = \rho(K,W,F_K) = (1/4,1/8,1/8,1/4),$$ which can be calculated with theorem 10.4 of \cite{Bn}.

The product $$\Psi_z(Z) = \mathbf{e}\Big( \langle \rho,Z \rangle \Big) \prod_{\substack{\lambda \in K' \\ \langle \lambda, W \rangle > 0 }} \Big( 1 - \mathbf{e}((\lambda, Z)) \Big)^{c(q(\lambda),\lambda)}$$ has singular weight, and therefore its Fourier expansion has the form $$\Psi_z(Z) = \sum_{\substack{\lambda \in K' \\ \langle \lambda, W \rangle > 0}} a(\lambda) \mathbf{e}\Big( \langle \lambda + \rho, Z \rangle \Big)$$ where $a(\lambda) = 0$ unless $\lambda + \rho$ has norm $0$. Since $\Psi_z(w(Z)) = \mathrm{det}(w) \Psi_z(w)$ for all elements of the Weyl group $w \in G,$ we can write this as $$\Psi_z(w(Z)) = \sum_{w \in G} \mathrm{det}(w) \sum_{\substack{\lambda \in K' \\ \lambda + \rho \in \overline{W} \\ \langle \lambda, W \rangle > 0}} a(\lambda) \mathbf{e}\Big( \langle w(\lambda + \rho), Z \rangle \Big).$$ As in \cite{DHM}, any such $\lambda$ must be a positive integer multiple of $\rho$; and in fact to be in $K'$ it must be a multiple of $4\rho$. Also, the only terms in the product that contribute to $a(\lambda)$ come from other positive multiples of $4 \rho$; i.e. $$\mathbf{e}\Big( \langle \rho, Z \rangle \Big) \prod_{m > 0} \Big[ 1 - \mathbf{e} \Big( \langle 4m \rho, Z \rangle \Big) \Big]^{c(0,4m\rho)} = \sum_{\substack{\lambda \in K' \\ \langle \lambda, W \rangle > 0 \\ \lambda + \rho \in \overline{W}}} a(\lambda) \mathbf{e} \Big( \langle \lambda + \rho, Z \rangle \Big).$$ Here, $c(0,4m \rho) = 2 \cdot (-1)^m$, so $$\sum_{\lambda} a(\lambda) \mathbf{e} \Big( \langle \lambda + \rho, Z \rangle \Big) = \mathbf{e}\Big( \langle \rho, Z \rangle \Big) \prod_{m > 0} \Big[ 1 - \mathbf{e} \Big( \langle 4m \rho, Z \rangle \Big) \Big]^{2 (-1)^m},$$ so we get the identity $$\Psi_z(Z) = \mathbf{e} \Big( \langle \rho,Z \rangle \Big) \prod_{\substack{\lambda \in K' \\ \langle \lambda, W \rangle > 0}} \Big( 1 - \mathbf{e} (\langle \lambda, Z \rangle ) \Big)^{c(q(\lambda),\lambda)} = \sum_{w \in G} \mathrm{det}(w) \mathbf{e} \Big( \langle w(\rho),Z \rangle \Big) \prod_{m=1}^{\infty} \Big[ 1 - \mathbf{e} \Big( 4m \langle w(\rho), Z \rangle \Big) \Big]^{2 (-1)^m}.$$ Note that the product on the right is an eta product $$q \prod_{m=1}^{\infty} \Big[ 1 - q^{4m} \Big]^{2 (-1)^m} = \frac{\eta(8\tau)^4}{\eta(4 \tau)^2},$$ so we can write this in the more indicative form $$\mathbf{e}\Big( \langle \rho,Z \rangle \Big) \prod_{\substack{\lambda \in K' \\ \langle \lambda, W \rangle > 0}} \Big( 1 - \mathbf{e}( \langle \lambda, Z \rangle ) \Big)^{c(q(\lambda),\lambda)} = \sum_{w \in G} \mathrm{det}(w) \frac{\eta(8 \langle w(\rho), Z \rangle)^4}{\eta(4 \langle w(\rho), Z \rangle)^2}.$$

\section{Example - computing Petersson scalar products}

One side effect of the computation of Poincar\'e square series is another way to compute the Petersson scalar product of (vector-valued) cusp forms numerically. This is rather easy so we will only give an example, rather than state a general theorem. Consider the weight $3$ cusp form \begin{align*} \Theta(\tau) &= \sum_{n, \gamma} c(n,\gamma) q^n \mathfrak{e}_{\gamma} \\ &= \Big( q^{1/6} + 2 q^{7/6} - 22q^{13/6} + 26q^{19/6} + ... \Big) (\mathfrak{e}_{(1/6,2/3)} + \mathfrak{e}_{(1/3,5/6)} + \mathfrak{e}_{(2/3,1/6)} + \mathfrak{e}_{(5/6,1/3)} - 2 \mathfrak{e}_{(1/6,1/6)} - 2 \mathfrak{e}_{(5/6,5/6)}) \\ &+ \Big(  -6 q^{1/2} + 18 q^{3/2} + 0 q^{5/2} - 12 q^{7/2} - ... \Big) (\mathfrak{e}_{(1/2,0)} + \mathfrak{e}_{(0,1/2)} - 2 \mathfrak{e}_{(1/2,1/2)}), \end{align*} which is the theta series with respect to a harmonic polynomial for the lattice with Gram matrix $\begin{pmatrix} -4 & -2 \\ -2 & -4 \end{pmatrix}.$ The component functions are $$q^{1/6} + 2 q^{7/6} - 22q^{13/6} + 26q^{19/6} + ... = \eta(\tau / 3)^3 \eta(\tau)^3 + 3 \eta(\tau)^3 \eta(3\tau)^3$$ and $$-6q^{1/2} + 18q^{3/2} + 0 q^{5/2}- 12q^{7/2} + ... = -6 \eta(\tau)^3 \eta(3\tau)^3.$$ To compute the Petersson scalar product $( \Theta, \Theta ),$ we write $\Theta$ as a linear combination of Eisenstein series and Poincar\'e square series; for example, $$\Theta = E_{3,0} - Q_{3,1/6,(1/6,1/6)}.$$ It follows that \begin{align*} (\Theta, \Theta ) &= -(\Theta, Q_{3,1/6,(1/6,1/6)} ) \\ &= -\frac{9}{2 \pi^2} \sum_{\lambda = 1}^{\infty} \frac{c(\lambda^2 / 6, (\lambda/6, \lambda/6))}{\lambda^4} \\ &= \frac{9}{\pi^2} \Big[ \sum_{\lambda \equiv 1,5 \, (6)} \frac{a(\lambda^2 / 2)}{\lambda^4} - 6\sum_{\lambda \equiv 3 \, (6)} \frac{a(\lambda^2 / 2)}{\lambda^4} \Big], \end{align*} where $a(n)$ is the coefficient of $n$ in $\eta(\tau)^3 \eta(3\tau)^3$. This series converges rather slowly but summing the first $150$ terms seems to give the value $( \Theta, \Theta ) \approx 0.24$. We get far better convergence for larger weights. \\

For scalar-valued forms (i.e. when the lattice $\Lambda$ is unimodular), applying this method to Hecke eigenforms gives the same result as a well-known method involving the symmetric square $L$-function. For example, the discriminant $$\Delta = q - 24 q^2 + ... = \sum_{n=1}^{\infty} c(n) q^n \in S_{12}$$ can be written as $$\Delta = \frac{53678953}{304819200} (Q_{12,1,0} - E_{12})$$ which gives the identity $$( \Delta, \Delta ) = \frac{131 \cdot 593 \cdot 691}{2^{23} \cdot 3 \cdot 7 \cdot \pi^{11}} \sum_{n=1}^{\infty} \frac{c(n^2)}{n^{22}}.$$ This identity is equivalent to the case $s = 22$ of equation (29) of \cite{Zag}: $$\sum_{n=1}^{\infty} \frac{c(n)^2}{n^{22}} = \frac{7 \cdot 11 \cdot 4^{22} \cdot \pi^{33} \cdot \zeta(11)}{2 \cdot 23 \cdot 691 \cdot 22! \cdot \zeta(22)} ( \Delta, \Delta ),$$ since $$\sum_{n=1}^{\infty} \frac{c(n)^2}{n^{22}} = \zeta(11) \sum_{n=1}^{\infty} \frac{c(n^2)}{n^{22}},$$ which can be proved directly using the fact that $\Delta$ is a Hecke eigenform.

\section{Appendix - averaging operators}

For applications to automorphic products, we do not need the Eisenstein series $E_{k,\beta}$ for any nonzero $\beta \in \Lambda'/\Lambda$ with $q(\beta) \in \mathbb{Z}$. This is essentially because the constant terms $\mathfrak{e}_{\beta}$, $\beta \ne 0$ are not counted towards the principal part of the input function $F$ in Borcherds' lift. However, the $E_{k,\beta}$ are still necessary in order to span the full space of modular forms. \\

It seems difficult to apply the formula for $E_{k,\beta}$ in \cite{BK} directly since the Kloosterman sums there do not reduce to Ramanujan sums. A brute-force way to find $E_{k,\beta}$ is to search for $\Delta \cdot E_{k,\beta}$ as a linear combination of Poincar\'e square series, but this is usually messy. Instead, we mention here that averaging over Schr\"odinger representations allows one in many (but not all) cases to read off the coefficients of all $E_{k,\beta}$ from those of $E_{k,0}$.

\begin{defn} Let $\beta \in \Lambda'/\Lambda$ have denominator $d_{\beta}$. The \textbf{averaging operator} attached to $\beta$ is \begin{align*} A_{\beta} : M_k(\rho^*) &\rightarrow M_k(\rho^*) \\ A_{\beta} F(\tau) &= \frac{1}{d_{\beta^2}} \sum_{\lambda,\mu \in \mathbb{Z}/d_{\beta}^2 \mathbb{Z}} \sigma_{\beta}^*(\lambda,\mu,0) F(\tau). \end{align*}
\end{defn}

This is well-defined because $\sigma_{\beta}^*(\lambda,\mu,0)$ depends only on the remainder of $\lambda$ and $\mu$ mod $d_{\beta}^2$; and it defines a modular form because $$\Big( \sigma_{\beta}^*(\zeta) F \Big) |_{k,\rho^*} M (\tau) = \sigma_{\beta}^*(\zeta \cdot M) F(\tau)$$ for all $\zeta \in \mathcal{H}$ and $M \in \tilde \Gamma.$ \\

Explicitly, if the components of $F$ are written out as $$F(\tau) = \sum_{\gamma \in \Lambda'/\Lambda} f_{\gamma}(\tau) \mathfrak{e}_{\gamma},$$ then $$A_{\beta} F(\tau) = \frac{1}{d_{\beta}^2} \sum_{\gamma \in \Lambda'/\Lambda} \sum_{\lambda \in \mathbb{Z}/d_{\beta}^2 \mathbb{Z}} \sum_{\mu \in \mathbb{Z}/d_{\beta}^2 \mathbb{Z}} \mathbf{e}\Big( -\mu \langle \beta, \gamma \rangle + \lambda \mu \cdot q(\beta) \Big) f_{\gamma}(\tau) \mathfrak{e}_{\gamma - \lambda \beta}.$$ The sum over $\mu$ is nonzero exactly when $\langle \beta, \gamma \rangle - \lambda q(\beta) \in \mathbb{Z}$, in which case it becomes $d_{\beta}^2$; therefore, $$A_{\beta} F(\tau) = \sum_{\lambda \in \mathbb{Z}/d_{\beta}^2 \mathbb{Z}} \sum_{\substack{\gamma \in \Lambda'/ \Lambda \\ \langle \beta, \gamma \rangle - \lambda q(\beta) \in \mathbb{Z}}} f_{\gamma}(\tau) \mathfrak{e}_{\gamma - \lambda \beta}.$$

In the special case that $q(\beta) \in \mathbb{Z}$, this is a constant multiple of the modified averaging operator $$A'_{\beta} F(\tau) = \sum_{\substack{\gamma \in \Lambda'/\Lambda \\ \langle \beta, \gamma \rangle \in \mathbb{Z}}} \Big( \sum_{\lambda \in \mathbb{Z}/d_{\beta} \mathbb{Z}} f_{\gamma + \lambda \beta}(\tau) \Big) \mathfrak{e}_{\gamma},$$ making this easier to compute. \\

When $F = E_{k,0}$ is the Eisenstein series $$E_{k,0}(\tau) = \frac{1}{2} \sum_{c,d} (c \tau + d)^{-k} \rho^*(M)^{-1} \mathfrak{e}_0,$$ then we get \begin{align*} A_{\beta} E_{k,0}(\tau) &= \frac{1}{2 d_{\beta}^2} \sum_{c,d} (c \tau + d)^{-k} \sum_{\lambda,\mu \in \mathbb{Z}/d_{\beta}^2 \mathbb{Z}} \sigma_{\beta}^* (\lambda,\mu,0) \rho^*(M)^{-1} \mathfrak{e}_0 \\ &= \frac{1}{2 d_{\beta}^2} \sum_{c,d} (c \tau + d)^{-k} \sum_{\lambda, \mu \in \mathbb{Z}/d_{\beta}^2 \mathbb{Z}} \mathbf{e}\Big( \lambda \mu q(\beta) \Big) \rho^*(M)^{-1} \mathfrak{e}_{-\lambda \beta} \\ &= \sum_{\substack{\lambda \in \mathbb{Z}/d_{\beta}^2 \mathbb{Z} \\ \lambda q(\beta) \in \mathbb{Z}}} E_{k,\lambda \beta}. \end{align*} In many cases this makes it possible to find all the Eisenstein series $E_{k,\beta}.$

\begin{ex} Let $S$ be the Gram matrix $S = \begin{pmatrix} -8 \end{pmatrix}.$ The Eisenstein series of weight $5/2$ is \begin{align*} E_{5/2,0}(\tau) &= \Big( 1 - 24q - 72q^2 - 96q^3 - ... \Big) \mathfrak{e}_{0} \\ &+ \Big( -\frac{1}{2}q^{1/16} - 24q^{17/16} - 72 q^{33/16} - ... \Big) (\mathfrak{e}_{1/8} + \mathfrak{e}_{7/8}) \\ &+ \Big( -5 q^{1/4} - 24q^{5/4} - 125q^{9/4} - ... \Big) (\mathfrak{e}_{1/4} + \mathfrak{e}_{3/4} ) \\ &+ \Big( -\frac{25}{2} q^{9/16} - \frac{121}{2} q^{25/16} - 96 q^{41/16} - ... \Big) (\mathfrak{e}_{3/8} + \mathfrak{e}_{5/8}) \\ &+ \Big( -46 q - 48q^2 - 144q^3 - ... \Big) \mathfrak{e}_{1/2}. \end{align*} Averaging over the Schr\"odinger representation attached to $\beta = (1/2)$ gives \begin{align*} E_{5/2,0}(\tau) + E_{5/2,1/2}(\tau) &= \Big( 1 - 70q - 120q^2 - 240q^3 - ... \Big) (\mathfrak{e}_0 + \mathfrak{e}_{1/2}) \\ &+ \Big( -10q^{1/4} - 48q^{5/4} - 250q^{9/4} - ... \Big) (\mathfrak{e}_{1/4} + \mathfrak{e}_{3/4}), \end{align*} from which we can read off the Fourier coefficients of $E_{5/2,1/2}(\tau).$
\end{ex}

\section{Appendix - calculating the Euler factors at $p = 2$}

We will summarize the calculations of Appendix B in \cite{CKW} as they apply to our situation.

\begin{prop} Let $f(X) = \bigoplus_{i \in \mathbb{N}_0} 2^i Q_i(X) \oplus L + c$ be a $\mathbb{Z}_2$-integral quadratic polynomial in normal form, and assume that all $Q_i$ are given by $Q_i(v) = v^T S_i v$ for a symmetric (not necessarily even) $\mathbb{Z}_2$-integral matrix $S_i$. For any $j \in \mathbb{N}_0$, define $$\mathbf{Q}_{(j)} := \bigoplus_{\substack{0 \le i \le j \\ i \equiv j \, (2)}} Q_i, \; \; \mathbf{r}_{(j)} = \mathrm{rank}(\mathbf{Q}_{(j)}), \; \; \mathbf{p}_{(j)} = 2^{\sum_{0 \le i < j} \mathbf{r}_{(i)}}.$$ Let $\omega \in \mathbb{N}$ be such that $Q_i = 0$ for all $i > \omega.$ Then: \\ (i) If $L = 0$ and $c = 0$, let $r = \sum_i \mathrm{rank}(Q_i)$; then the Igusa zeta function for $f$ at $2$ is \begin{align*} \zeta_{Ig}(f;2;s) &= \sum_{0 \le \nu < \omega - 1} \frac{2^{-\nu s}}{\mathbf{p}_{(\nu)}} I_0(\mathbf{Q}_{(\nu)}, \mathbf{Q}_{(\nu+1)}, Q_{\nu+2}) \; + \\ &\quad \quad \quad + \Big[ \frac{2^{-s(\omega - 1)}}{\mathbf{p}_{(\omega - 1)}} I_0(\mathbf{Q}_{(\omega-1)}, \mathbf{Q}_{(\omega)}, 0) + \frac{2^{-\omega s}}{\mathbf{p}_{(\omega)}} I_0(\mathbf{Q}_{(\omega)}, \mathbf{Q}_{(\omega - 1)}, 0) \Big] \cdot (1 - 2^{-2s - r})^{-1}.\end{align*} (ii) If $L(x) = bx$ for some $b \ne 0$ with $v_2(b) = \lambda$ and if $v_2(b) \le v_2(c)$, then \begin{align*} \zeta_{Ig}(f;2;s) &= \sum_{0 \le \nu < \lambda - 2} \frac{2^{-\nu s}}{\mathbf{p}_{(\nu)}}  I_0(\mathbf{Q}_{(\nu)}, \mathbf{Q}_{(\nu+1)}, Q_{\nu+2}) +\\ &\quad\quad\quad + \sum_{\max\{0,\lambda - 2\} \le \nu < \lambda} \frac{2^{-\nu s}}{\mathbf{p}_{(\nu)}} I_0^{\lambda - \nu}(\mathbf{Q}_{(\nu)}, \mathbf{Q}_{(\nu+1)}, Q_{i+2}) + \frac{2^{-\lambda s}}{\mathbf{p}_{(\lambda)}} \cdot \frac{1}{2 - 2^{-s}}.\end{align*} (iii) If $L(x) = bx$ with $b \ne 0$ and $v_2(c) < v_2(b) \le v_2(c) + 2$, let $\kappa = v_2(c)$; then \begin{align*} \zeta_{Ig}(f;2;s) &= \sum_{0 \le \nu < \lambda - 2} \frac{2^{-\nu s}}{\mathbf{p}_{(\nu)}} I_{c / 2^{\nu}}(\mathbf{Q}_{(\nu)}, \mathbf{Q}_{(\nu + 1)}, Q_{\nu + 2}) \\ &\quad\quad \quad + \sum_{\max\{0,\lambda - 2\} \le \nu \le \kappa} \frac{2^{-\nu s}}{\mathbf{p}_{(\nu)}} I^{\lambda - \nu}_{c / 2^{\nu}}( \mathbf{Q}_{(\nu)}, \mathbf{Q}_{(\nu + 1)}, Q_{\nu + 2}) + \frac{1}{\mathbf{p}_{(\kappa + 1)}} 2^{-\kappa s}.\end{align*} (iv) If $L = 0$ or $L(x) = bx$ with $v_2(b) > v_2(c) + 2$, let $\kappa = v_2(c)$; then \begin{align*} \zeta_{Ig}(f;2;s) &= \sum_{0 \le \nu \le \kappa} \frac{2^{-\nu s}}{\mathbf{p}_{(\nu)}} I_{c / 2^{\nu}}(\mathbf{Q}_{(\nu)}, \mathbf{Q}_{(\nu + 1)}, Q_{\nu + 2}) + \frac{1}{\mathbf{p}_{(\nu + 1)}} 2^{-\kappa s}.\end{align*}
\end{prop}

Here, $I_a^b(Q_0,Q_1,Q_2)(s)$ are helper functions that we describe below, and we set $I_a(Q_0,Q_1,Q_2) = I^{\infty}_a(Q_0,Q_1,Q_2).$ Note that not every unimodular quadratic form $Q_i$ over $\mathbb{Z}_2$ can be written in the form $Q_i(v) = v^T S_i v$; but $2 \cdot Q_i$ can always be written in this form, and replacing $f$ by $2 \cdot f$ only multiplies $\zeta_{Ig}(f;2;s)$ by $2^{-s}$, so this does not lose generality. \\

Every unimodular quadratic form over $\mathbb{Z}_2$ that has the form $Q_i(v) = v^T S_i v$ is equivalent to a direct sum of at most two one-dimensional forms $a \cdot \mathrm{Sq}(x) = ax^2$; at most one elliptic plane $\mathrm{Ell}(x,y) = 2x^2 + 2xy + 2y^2$; and any number of hyperbolic planes $\mathrm{Hyp}(x,y) = 2xy.$ This decomposition is not necessarily unique. It will be enough to fix one such decomposition. \\

The following proposition explains how to compute $I_a^b(Q_0,Q_1,Q_2)(s).$ \\

\begin{prop} Define the function $$\mathrm{Ig}(a,b,\nu) = \begin{cases} \frac{2^{-\nu s}}{2 - 2^{-s}}: & v_2(a) \ge \mathrm{min}(b,\nu); \\ 2^{-v_2(a) s}: & v_2(a) < \mathrm{min}(b,\nu). \end{cases}$$ (Here, $v_2(0) = \infty$.) For a unimodular quadratic form $Q$ of rank $r$, fix a decomposition into hyperbolic planes, at most one elliptic plane and at most two square forms as above. Let $\varepsilon=1$ if $Q$ contains no elliptic plane and $\varepsilon = -1$ otherwise. Define functions $H_1(a,b,Q)$, $H_2(a,b,Q)$ and $H_3(a,b,Q)$ as follows: \\

\noindent (i) If $Q$ contains no square forms, then \begin{align*} H_1(a,b,Q) &= (1 - 2^{-r}) \mathrm{Ig}(a,b,1); \\ H_2(a,b,Q) &= \Big( 1 - 2^{-r/2} \varepsilon \Big) \cdot \Big( \mathrm{Ig}(a,b,1) +  2^{-r/2}\varepsilon \mathrm{Ig}(a,b,2) \Big); \\ H_3(a,b,Q) &= 0. \end{align*} (ii) If $Q$ contains one square form $cx^2$, then \begin{align*} H_1(a,b,Q) &= \mathrm{Ig}(a,b,0) - 2^{-r} \mathrm{Ig}(a,b,1); \\ H_2(a,b,Q) &= (1 - 2^{-(r-1)/2} \varepsilon) \mathrm{Ig}(a,b,0) - 2^{-r} \mathrm{Ig}(a,b,2) + 2^{-(r+1)/2}\varepsilon (\mathrm{Ig}(a,b,2) + \mathrm{Ig}(a+c,b,2)); \\ H_3(a,b,Q) &= 2^{-r} (\mathrm{Ig}(a+c,b,3) - \mathrm{Ig}(a+c,b,2)). \end{align*} (iii) If $Q$ contains two square forms $cx^2$, $dx^2$ and $c + d \equiv 0 \, (4)$, then \begin{align*} H_1(a,b,Q) &= \mathrm{Ig}(a,b,0) - 2^{-r} \mathrm{Ig}(a,b,1); \\ H_2(a,b,Q) &= \mathrm{Ig}(a,b,0) - 2^{-r/2}\varepsilon \mathrm{Ig}(a,b,1) + (2^{-r/2}\varepsilon - 2^{-r}) \mathrm{Ig}(a,b,2); \\ H_3(a,b,Q) &= (-1)^{(c+d)/4} 2^{-r} (\mathrm{Ig}(a,b,3) - \mathrm{Ig}(a,b,2)). \end{align*} (iv) If $Q$ contains two square forms $cx^2$, $dx^2$ and $c + d \not \equiv 0 \, (4)$, then \begin{align*} H_1(a,b,Q) &= \mathrm{Ig}(a,b,0) - 2^{-r} \mathrm{Ig}(a,b,1); \\ H_2(a,b,Q) &=  (1 - 2^{-(r-2)/2} \varepsilon) \mathrm{Ig}(a,b,0) + 2^{-r/2} \varepsilon (\mathrm{Ig}(a,b,1) + \mathrm{Ig}(a+c,b,2)) - 2^{-r} \mathrm{Ig}(a,b,2); \\ H_3(a,b,Q) &= -2^{1-r} \mathrm{Ig}(a,b,1) + 2^{-r} (\mathrm{Ig}(a,b,2) + \mathrm{Ig}(a+c+d,b,3) ). \end{align*}

Let $\varepsilon_1 = 1$ if $Q_1$ contains no elliptic plane and $\varepsilon_1 = -1$ otherwise, and let $r_1$ denote the rank of $Q_1$. Then $I_a^b(Q_0,Q_1,Q_2)$ is given as follows: \\ (1) If both $Q_1$ and $Q_2$ contain at least one square form, then $$I_a^b(Q_0,Q_1,Q_2) = H_1(a,b,Q_0).$$ (2) If $Q_1$ contains no square forms but $Q_2$ contains at least one square form, then $$I_a^b(Q_0,Q_1,Q_2) = H_2(a,b,Q_0).$$ (3) If both $Q_1$ and $Q_2$ contain no square forms, then $$I_a^b(Q_0,Q_1,Q_2) = H_2(a,b,Q_0) + 2^{-r_1 / 2} \varepsilon_1  H_3(a,b,Q_0).$$ (4) If $Q_1$ contains one square form $cx^2$, and $Q_2$ contains no square forms, then $$I_a^b(Q_0,Q_1,Q_2) = H_1(a,b,Q_0) + 2^{-(r_1 + 1)/2} \varepsilon_1 (H_3(a,b,Q_0) + H_3(a+2c,b,Q_0)).$$ (5) If $Q_1$ contains two square forms $cx^2$ and $dx^2$ such that $c+d \equiv 0\, (4)$, and $Q_2$ contains no square forms, then $$I_a^b(Q_0,Q_1,Q_2) = H_1(a,b,Q_0) + 2^{-r_1 / 2} \varepsilon_1 H_3(a,b,Q_0).$$ (6) If $Q_1$ contains two square forms $cx^2$ and $dx^2$ such that $c +d \not \equiv 0 \, (4)$, and $Q_2$ contains no square forms, then $$I_a^b(Q_0,Q_1,Q_2) = H_1(a,b,Q_0) + 2^{-r_1/2} \varepsilon_1 H_3(a+c,b,Q_0).$$
\end{prop}
\begin{proof} In the notation of \cite{CKW}, $$\mathrm{Ig}(a,b,\nu) = \mathrm{Ig}(z^{a + 2^b \mathbb{Z}_2 + 2^{\nu} \mathbb{Z}_2})$$ and $$H_1(a,b,Q) = \mathrm{Ig}\Big(z^{a + 2^b \mathbb{Z}_2} \hat H_Q(z)\Big)$$ and $$H_2(a,b,Q) = \mathrm{Ig}\Big(z^{a + 2^b \mathbb{Z}_2} \tilde H_Q(z)\Big)$$ and $$H_3(a,b,Q) = \mathrm{Ig}\Big(z^{a + 2^b \mathbb{Z}_2} (H_Q(z) - \tilde H_Q(z))\Big).$$ This calculation of $I_a^b(Q_0,Q_1,Q_2)$ is available in Appendix B of \cite{CKW}. Finally, the calculation of $\zeta_{Ig}(f;2;s)$ is given in theorem 4.5 loc$.$ cit.
\end{proof}

\bibliographystyle{plainnat}
\bibliography{\jobname}

\end{document}